\documentclass[a4paper,reqno,12pt]{amsart}

\usepackage[utf8]{inputenc}
\usepackage{pslatex}             
\usepackage{amssymb}
\usepackage{latexsym}
\usepackage{graphicx}
\usepackage[mathscr]{euscript}
\usepackage{url}
\usepackage{color}
\usepackage{harpoon}
\usepackage{tikz}
\usepackage{newtxmath}    
\usepackage{hyperref}
\usepackage{float}

\usepackage{fullpage} 
\usepackage{setspace}
\usepackage{adjustbox}
\usepackage{mathtools}
\allowdisplaybreaks
\mathtoolsset{showonlyrefs}

\def\today{\ifcase\month\or
  January\or February\or March\or April\or May\or June\or
  July\or August\or September\or October\or November\or December\fi
  \space\number\day, \number\year}

\theoremstyle{plain}
\newtheorem{theorem}{Theorem}

\newtheorem{lemma}[theorem]{Lemma}
\newtheorem{proposition}[theorem]{Proposition}

\theoremstyle{plain}
\newtheorem{definition}[theorem]{Definition}
\newtheorem*{problem}{Rudin's Problem}
\newtheorem*{remark*}{Remark}
\theoremstyle{remark}

\theoremstyle{plain}

\newcommand{\B}{\mathscr{B}}
\newcommand{\C}{\mathscr{C}}
\newcommand{\D}{\mathscr{D}}
\newcommand{\E}{\mathscr{E}}
\newcommand{\F}{\mathscr{F}}

\newcommand{\PolySpace}{\mathrm{H}}
\renewcommand{\H}{\mathrm{H}}

\newcommand{\cP}{\mathscr{P}}

\newcommand{\cS}{\mathscr{S}}

\newcommand{\n}{\mathbb{N}}
\newcommand{\z}{\mathbb{Z}}

\newcommand{\RR}{\mathbb{R}}
\newcommand{\cp}{\mathbb{C}} 

\newcommand{\ft}{\widehat}

\newcommand{\bo}{\boldsymbol}

\newcommand{\la}{\lambda}

\newcommand{\al}{\alpha}

\newcommand{\si}{\sigma}

\newcommand{\ov}{\overline}
\newcommand{\dd}{\mathrm{d}}

\newcommand{\p}{\varphi}

\newcommand{\aff}{\text{Aff}}
\newcommand{\GL}{\operatorname{GL}}
\newcommand{\Sphere}{\mathbb{S}}

\newcommand{\Z}{\mathbb{Z}}

\newcommand{\supp}{\operatorname{supp}}
\newcommand{\vol}{\operatorname{vol}}
\newcommand{\lcmk}{\operatorname{lcm}_k}
\newcommand{\len}{\operatorname{len}}
\newcommand{\sect}{\operatorname{sect}}

\newcommand{\mymod}[1]{{\, \, ({\rm mod} \ \, #1)}}

\theoremstyle{plain}
\newtheorem*{definition*}{Definition}
\theoremstyle{plain}
\providecommand{\widebar}{\overline}
\begin{document}

\title[]{Uniformly Bounded Orthonormal Bases of Homogeneous Polynomials on the Complex Sphere}
\author{Felipe Gon\c{c}alves}
\address{IMPA - Instituto de Matemática Pura e Aplicada, Rio de Janeiro, 22460-320, Brazil.}
\email{goncalves@impa.br}

\author{João P. Ferreira}
\address{IMPA - Instituto de Matemática Pura e Aplicada, Rio de Janeiro, 22460-320, Brazil.}
\email{joao.ferreira@impa.br}

\author{Maksym Radziwi\l{}\l{}}
\address{Courant Institute of Mathematical Sciences, New York University,
251 Mercer Street, New York, NY 10012, USA.}
\email{maksym.radziwill@gmail.com}

\date{\today}
\allowdisplaybreaks

\begin{abstract}
We resolve the long-standing problem of Rudin concerning uniformly bounded orthonormal bases of homogeneous polynomials on complex unit spheres. Specifically, we show that for every $K \geq 2$ and $N \geq 0$ there exists a basis of degree $N$ homogeneous polynomials in $\mathbb{C}^K$ that is orthonormal with respect to the surface measure of the unit sphere of $\mathbb{C}^K$ and uniformly bounded by a constant depending only on $K$ on the unit ball of $\mathbb{C}^K$. Bourgain resolved the case $K = 2$ in 1985 and the case $K = 3$ in 2016. No cases $K \geq 4$ were previously known.
\end{abstract}

\maketitle


\section{Introduction}\label{sec:intro}
{
Let  $\mathbb B_K :=\{z\in\cp^K:
 |z_1|^2+\cdots+|z_K|^2<1\}$ be the complex unit ball and $\Sphere^{2K-1}=\partial {\mathbb B_K}$ be the complex unit sphere. Let $H^2(\mathbb B_K)$ {denote} the Hardy space of holomorphic functions} $f:\mathbb B_K\to \cp$ that have radial boundary values in $L^2$, that is, $f(r\,\cdot)$ converges in
$L^2(\Sphere^{2K-1})$ as $r\to1^-$. {In this way, the Hardy space $H^2(\mathbb B_K)$ inherits the Hilbert space structure of $L^2(\Sphere^{2K-1})$ (see \cite{R}). Write $z = (z_1, \ldots, z_K)$ and let
\[
\PolySpace_N :=\operatorname{span}_{\cp}
 \left\{z^j:j\in\z_{\geq 0}^K, \ j_1+\cdots+j_K=N\right\} \quad (z^{j} = z_1^{j_1} \cdots z_K^{j_K})
\]
denote the space of homogeneous polynomials of degree $N$ (note that the dimension $K$ is implied from the context). A routine calculation shows that $z^j$ and $z^l$ with $j\neq l$ are orthogonal in $H^2(\mathbb B_K)$. In particular, we have the following orthogonal decomposition
\[
 H^2(\mathbb B_K)=\widebar{\bigoplus_{N\geq0}\PolySpace_N}.
\]
The following problem is known in the literature as ``Rudin's problem'' (see \cite[p.~61]{GGSTW}), famously proposed by him at a conference at Princeton University in 1983:
\begin{problem}
Let $K \geq 2$.
Can one construct, for every $N\geq 0$, an orthonormal basis
$\{\phi_{1,N},\phi_{2,N}, \ldots, \phi_{\dim \PolySpace_N, N}\}$ 
of $\PolySpace_N$ such that
\[
\sup_{\substack{N \geq 0 \\ 1 \leq i \leq \dim \PolySpace_N}} \| \phi_{i, N} \|_{L^{\infty}(\mathbb B_K)} < \infty
 \ ?
\]
\end{problem}

An important precursor of this question is the work of Ryll and Wojtaszczyk \cite{RW} in 1983, where they constructed for each $N \geq 0$ a polynomial $P_N \in \PolySpace_N$
such that 
$$
\| P_N \|_{L^{\infty}(\mathbb B_K)} \leq C_K \| P_N \|_{L^2(\Sphere^{2K - 1})}
$$
with $C_K > 0$ depending only on $K$. Combined with the work of Olevskii
\cite[Chapter~4]{OlevskiiBook} and Ovsepian--Pe{\l}czy\'nski \cite{OvsepianPelczynski},
this yields for each $K \geq 2$ a uniformly bounded orthonormal basis of $H^2(\mathbb{B}_K)$ consisting of polynomials. However, these polynomials are not homogeneous.  Another construction of a uniformly bounded basis of $H^2(\mathbb{B}_K)$ is due to Rudin
\cite[Appendix~I]{RudinNew}. This relies on Aleksandrov's and L\"ow's solution to the inner function problem 
(see \cite{AleksandrovInner}, \cite{Low}). The construction is non-explicit and its
basis elements are not polynomials. 


Rudin's problem is known to have an affirmative answer when $K = 2$ and
$K = 3$. Both cases are due to Bourgain, respectively in 1985 \cite{Bour1} and
2016 \cite{Bour2}. In higher dimension $K \geq 4$, Shiffman \cite{S} constructed for 
each $N > 0$ an orthonormal uniformly bounded system of cardinality
at least $\beta_K \dim \PolySpace_N$ with $\beta_K > 0$. Marzo and Ortega-Cerd\`a \cite{MO}
subsequently strengthened this to at least
$(1-\varepsilon)\dim\PolySpace_N$ for every $\varepsilon>0$ with a 
uniform bound depending on $K$ and $\varepsilon$. Neither result gives a complete basis.  
 
Our main result is a resolution of Rudin's problem.  

\begin{theorem}\label{thm:onb}
Let $K\geq 2$. Then for every $N\geq 0$
the space $\PolySpace_N$ admits an orthonormal basis
$\{\p_{1, N},\p_{2, N},\ldots,\p_{\dim\PolySpace_N, N}\}$ such that
\[
  \sup_{\substack{N \geq 0 \\ 1\leq i\leq \dim\PolySpace_N}}\|\p_{i, N}\|_{L^\infty(\mathbb B_K)}<\infty.
\]
\end{theorem}
We now make a few comments.

Bourgain asked (\cite[p.~957]{RememberingBourgain}) if one can
place $n$ points in the unit square so that the $1/n$-neighborhood of every
line contains at most an absolute constant number of them.  Bourgain's
motivation was to combine Gaussian beams into uniformly bounded spherical
harmonics. While Theorem~\ref{thm:onb} settles Rudin's problem, it yields
no new information on Bourgain's question.
 
A homogeneous holomorphic polynomial on $\cp^K\simeq\RR^{2K}$ is a
complex-valued spherical harmonic on $\Sphere^{2K-1}$.  One can therefore ask the analogue of Rudin's question on any real sphere $\Sphere^{M}$, with real spherical harmonics of degree $N$ and inner product inherited from the surface area. In this setting the problem is open for any $M\geq 2$. Astonishingly, even the existence of a sequence of harmonics $\{p_N\}_{N\geq 0}$, each homogeneous of degree $N$, and such that $\|p_N\|_{L^\infty(\Sphere^{2K})} \ll_K \|p_N\|_{L^2(\Sphere^{2K})}$ is not known \cite[Corollary~1.4]{S} (the single harmonic case in $\Sphere^{2K-1}$ follows from the work of Ryll and Wojtaszczyk \cite{RW}).

We now {briefly discuss} Bourgain's construction in the cases $K = 2$ and $K = 3$, as it is relevant to our construction. We start with the case $K = 2$. Bourgain's basis is obtained by setting
\[
 \p_l^N(z_1,z_2)=
 \sum_{j=0}^N\sigma_j e\left(\frac{jl}{N+1}\right)\binom{N}j^{1/2} z_1^j z_2^{N-j} \qquad 
(e(z)=e^{2\pi i z})
\]
for $0\leq l\leq N$ and making an appropriate choice of unimodular coefficients $\si_j \in \Sphere^1$. Bourgain chooses $\si_j$ to be the Golay--Rudin--Shapiro signs\footnote{These coefficients appeared in disguised form in Golay's work on
spectroscopy, were found independently in Shapiro's thesis, and were
later rediscovered and published by Rudin
\cite{Golay,Shapiro,RudinRS}.}.  However, the only property that the unimodular coefficients $(\si_j)_{j\geq 0}$ need to possess is exact
square-root cancellation\footnote{{We use the expression} $A \ll_\Lambda B$ {to mean that there is $c>0$}, depending only on the set of parameters $\Lambda$, such that $A \leq c B$. If $\Lambda$ is omitted the constant is numeric. We also write $A \asymp_\Lambda B$ to express that both $A \ll_\Lambda B$ and $B \ll_\Lambda A$ hold.}
\begin{equation} \label{eq:sqrt_intro}
 \sup_{\theta\in\RR}
 \bigg|\sum_{j\in I}\sigma_j e(j\theta)\bigg|
 \ll(\#I)^{1/2}
\end{equation}
for every integer interval $I$. Indeed, if we write $z_1=r^{1/2}{e(\theta)}$ with $0<r<1/2$, Stirling’s inequalities imply that the weights in $\p_{l}^{N}$ are concentrated in $Nr + O(1 + \sqrt{N r})$ with exponential decay outside of this region. Specifically,
\[
\binom{N}j^{1/2}r^{j/2}(1-r)^{(N-j)/2} \ll  H^{-1/2} \exp \Big ( - \frac{1}{4} \cdot \frac{|j - N r|}{H} \Big )  \ \ \text{with} \ \ H := (1 + N r)^{1/2}.
\]
Therefore, using summation by parts we see that \eqref{eq:sqrt_intro} suffices to conclude that $\p_l^{N}$ is bounded. 
There are in fact many sequences $\si_j$ that have the property \eqref{eq:sqrt_intro}, although we note that they are not ``generic'': a sequence of random independent and identically distributed signs does not have this property. A theorem of Hardy and Littlewood \cite[Theorems~2.141 and~2.181]{HL} shows that when $\alpha$ is an irrational with bounded partial quotients, e.g., when $\alpha = \sqrt{d}$ with $d$ a positive square-free integer, one has
\[
 \sup_{\theta\in\RR}
 \bigg|\sum_{j\in I}e(\alpha j^2+j\theta)\bigg|
 \ll_\alpha(\#I)^{1/2}.
\]
Consequently, Bourgain's $K=2$ construction works equally well with the choice $\sigma_j=e(\sqrt{2} j^2)$.  {In 2016 Bourgain}  \cite{Bour2} solved the full homogeneous-basis problem for $K=3$.  The exponent set is now the triangle $\{(j_1,j_2)\in\z_{\geq0}^2:j_1+j_2\leq N\}$, the third component $j_3$ being uniquely determined as $N - j_1 - j_2$.
Bourgain separates the edges with $j_1j_2=0$ and uses the previous basis there. On the remaining vertices he rearranges the triangle $j_1 + j_2 \leq N$ with $j_1 j_2 \neq 0$ into a rectangle by a direct ``cut and paste'' construction (to ensure orthogonality), and then exploits the analogue of \eqref{eq:sqrt_intro} for the sequence $e(\sqrt{2} (j_1^2 + j_2^2))$. Given the Hardy and Littlewood theorem mentioned above, it begs the question whether one can simply stick with the unimodular coefficients 
\[
e\bigl(\sqrt{2} (j_1^2 +\cdots+j_K^2)\bigr)
\]
generically in every dimension. {This is exactly what we will show can be accomplished.} Compared to Bourgain's construction one complication that emerges in higher dimension is that we will need a multi-variable analogue of \eqref{eq:sqrt_intro} over an arbitrary polytope. The analogue of the ``cut-and-paste'' construction (needed for orthogonality) increases in complexity and we can no longer map $j_1 + \cdots + j_K \leq N$ directly to a simple rectangular region, which makes the relevant exponential sum estimates harder.

\subsection{Notations and Definitions}

We will adhere to a set of conventions and definitions throughout the paper. 
Let $k \geq 1$.
For a finite set $L\subset\Z^k\times\Z$, whose elements we call
\emph{constraints}, put
\[
\cP(L):=\{n\in\Z^k : v\cdot n+b\ge0 \text{ for every }(v,b)\in L\}.
\]
A \emph{$\Z$-polytope} is a set of the form $\cP(L)$. 

A \emph{box} $J\subset \z^k$ is a product of finite intervals, i.e., 
$J=J_1\times\cdots\times J_k$ with $J_i = [a_i,b_i] \cap \mathbb{Z}$. 
For a box $J$ and $w = (w_1, \ldots, w_k) \in\z^k$, put
\[
  \len(w;J):=\max\Bigl(1,\max_{j : w_j \neq 0}\#J_j\Bigr).
\]
Thus $\len(w;J)$ is the longest side seen by $w$. Define
\[
\vol(J) := \prod_{i = 1}^{k} \# J_i \quad \text{and} \quad \sect(w;J):=\frac{\vol(J)}{\len(w;J)}.
\] 

Given an irrational number $\al$ write $\alpha := [a_0; a_1, a_2, \ldots]$ for the continued fraction
expansion of $\al$. We say that $\al$ has bounded partial quotients when there exists a $C > 0$ such
that $a_i \leq C$ for all $a_i$.

Finally, for any symbol $\Box$ whatsoever, whenever we write $\Box \in F^k$ with $F$ a set 
we will denote by $\Box_j$ the $j$th component of $\Box$ so that $\Box = (\Box_1, \ldots, \Box_k)$.   
Thus if $s \in \mathbb{Q}^k$ then we automatically infer that $s_2$ refers to the second component of $s$,
i.e., $s = (s_1, s_2, \ldots, s_k)$. For any vector $v \in F^k$ we will also denote by $\supp v$ the set of
indices $j$ for which $v_j \neq 0$. For a real vector $v$ we write $|v|_1=\sum_j|v_j|$ and
$|v|_\infty=\max_j|v_j|$, and for a real matrix $A$ we write $|A|_\infty$ for the largest absolute
value of an entry of $A$.

\subsection{Outline}\label{sec:outline}

We first start with the main novel {number-theoretic} contribution of this paper, which will be the engine behind our deterministic construction. 

\begin{theorem}\label{thm:box}
Let {$k\geq1$,} $r\ge 0$ and $V,h,a\ge1$. Let $\alpha = [a_0; a_1, a_2, \ldots]$ be
an irrational number with $a_i \leq a$ for all $i \geq 1$.  
Let $c_1,{\ldots},c_k$ be positive integers with
$c_j\le h$. Let $I=I_1\times\cdots\times I_k$
be a box of nonempty intervals. Let
$L$ be a set of at most $r$ pairs $(v,b)\in\mathbb Z^k\times
\mathbb Z$ with $|v|_1\le V$. Then
\[
\sup_{\theta \in \RR^k} \, \Bigl|\sum_{n\in \cP(L) \cap I}e\Bigl(\alpha\sum_{j=1}^kc_jn_j^2
+\theta\cdot n\Bigr)\Bigr|
\ \ll_{k,r,V,h,a}\ \vol(I)^{1/2}.
\]
If both $(w_0,b_0)$ and $(-w_0,-b_0)$ belong to $L$, with
$w_0\ne0$, then also
\[
\sup_{\theta \in \RR^k} \, \Bigl|\sum_{n\in \cP(L) \cap I}e\Bigl(\alpha\sum_{j=1}^kc_jn_j^2
+\theta\cdot n\Bigr)\Bigr|
\ \ll_{k,r,V,h,a}\ \sect(w_0;I)^{1/2}.
\]
\end{theorem}

Theorem \ref{thm:box} is relatively easy to prove when the box is \textit{balanced}, that is,
$\# I_i \asymp \# I_j$ for all $1 \leq i, j \leq k$. This is Theorem~\ref{thm:general}, proved in
Section~\ref{sec:quadratic}. In Section~\ref{sec:decomp} we prove a decomposition theorem which
reduces the general case to the balanced one by splitting, adding and subtracting the conditions
in $L$. The two are combined into a proof of Theorem \ref{thm:box} in Section~\ref{sec:proofbox}.

In Section~\ref{part:one} we construct the basis, conditional on Theorem~\ref{thm:box} and on an
\textit{equidecomposition} of the polytope $j_1 + \cdots + j_{K - 1} \leq N$ into a rectangular box
which we now describe. This is the higher-dimensional analogue of Bourgain's
``cut and paste'' construction. 
This construction is done only when
$$
\lcmk := \text{lcm}(1,2,\ldots, k)
$$
divides $N$. In the remaining cases such an equidecomposition is not needed as the basis can be constructed recursively using an absorption trick. Before we proceed further let us define the meaning of \textit{equidecomposition}. The following definition is inspired by the notions of equidecomposition in \cite{HNP,HM08,KKMS}, but adapted to our needs.

\begin{definition}\label{def:equicomp}
Let $\aff_k(\Z) = \GL_k(\Z) \ltimes \Z^k $ be the group of affine transformations of the form $x\in \RR^k \mapsto Ax+b \in \RR^k$, where $A \in \GL_k(\Z)$ and $b\in \Z^k$. Let $P,Q\subset \mathbb Z^k$ be $\Z$-polytopes. We say that $P$ and $Q$ are \emph{$\aff_k(\Z)$-equidecomposable with complexity bounded by $R \geq 1$} if there is a disjoint collection of $\Z$-polytopes $\cP(L_1), \ldots, \cP(L_r)$ and aﬃne transformations $A_1 x + b_1,\ldots, A_r x + b_r \in \aff_k(\z)$ such that
\[
P = \bigsqcup_{i=1}^r {\cP(L_i)} \quad \text{and}  \quad Q = \bigsqcup_{i=1}^r (A_i \cP(L_i) + b_i),
\]
where
\[
r\leq R \ , \ \forall i: \# L_i \leq R \ , \ \forall i: |A_i|_\infty \leq R \ \ \text{and}\ \  \forall i \forall (v,b) \in L_i:   |v|_1 \leq R.
\]
\end{definition}
Notice that whenever two polytopes $P$ and $Q$ are equidecomposable as above we can write
a bijection $U : P \rightarrow Q$, piecewise affine, and defined by 
$$
U_{|\cP(L_i)}(x) = A_i x + b_i.
$$
Define the following $\Z$-polytopes
\begin{align}\label{eq:originalSimplexAndBox}
\cS_N&:=\left\{j\in \Z_{\ge0}^k:
  j_1+\cdots+j_k\le N\right\} \quad \text{and} \quad
\B_N:=\prod_{s=1}^k \bigg\{0,1,\ldots, \frac{N}{k+1-s}\bigg\},
\end{align}
where $\B_N$ is only defined when $\lcmk \mid N$.

\begin{proposition}\label{prop:UNmap}
Let {$N\geq0$ with} $\lcmk \mid N$. Then the $\Z$-polytopes $\cS_N$ and $\B_N$ are $\aff_k(\Z)$-equidecomposable with complexity $\ll_k 1$. Specifically, for each $N$, there exists a decomposition, 
$$
\cS_N := \bigsqcup_{i=1}^r {\cP(L_i)}
$$
with $r \ll_k 1$, $\#L_i \ll_k 1$ for each $i$, and 
$$
\sup_{\substack{1 \leq i \leq r \\ (v,b) \in L_i}} |v|_1 \ll_k 1,
$$
and a bijection $U_N : \cS_N \to \B_N$ such that ${U_N}_{|\cP(L_i)} = A_i x + b_i$
with $|A_i|_{\infty} \ll_k 1$ for each $i$.
\end{proposition}

We prove the existence of $U_N$ in Section \ref{sec:explicitEquidecomposition} in two distinct ways: one implicit, but straightforward and another explicit and constructive, but a bit more involved. The benefit of the constructive proof is that it {provides an explicit recipe} to build the basis which we shall define below. The first one is a direct application of triangulations into open unimodular simplices and Ehrhart polynomial theory. Indeed, since the real polytopes $S=\{x\in {\RR_{\geq0}^k} : x_1+\cdots+x_k \leq \lcmk\}$ and  $B=\prod_{s=1}^k [0,\lcmk/(k+1-s)]$ have the same Ehrhart polynomial (i.e. $\vol(tS \cap \Z^k)=\vol(t B \cap \Z^k)$ for all $t\in \Z_{\geq 0}$) {and both admit unimodular triangulations,} they are equidecomposable by $ \aff_k(\z)$ transformations, which implies their respective integer sets are $\aff_k(\Z)$-equidecomposable with complexity bounded {independently} of $N$.  {We then use the identities $\cS_N = (\frac{N}{\lcmk}S)\cap\Z^k$ and $\B_N = (\frac{N}{\lcmk}B)\cap\Z^k$, and dilate the real pieces before intersecting with the lattice.}

Recall that {the} $H^2(\mathbb B_K)$ space is equipped with the following
inner product 
\begin{align*}
 \langle f,g\rangle
 :=
 \int_{\Sphere^{2K-1}}f(z)\ov{g(z)}\,\dd\sigma(z),
\end{align*}
where $\si$ is the induced Lebesgue surface measure on $\Sphere^{2K-1}$ normalized so that $\si (\Sphere^{2K-1})=1$.
A routine computation (see \cite{R}) shows that for any $j,l \in \z_{\geq 0}^K$ {with $N=j_1+\cdots+j_K=l_1+\cdots+l_K$} we have
\begin{equation}\label{def:innerprod}
 \langle z^j,z^l\rangle
 =\binom{N+K-1}{K-1}^{-1}\binom Nj^{-1}{\bo \delta}_{j,l},
 \quad \text{where} \quad
 \binom Nj:=\frac{N!}{j_1!\cdots j_K!}.
\end{equation}

The following theorem, proved in Section~\ref{part:one}, constructs the basis conditional on
Theorem~\ref{thm:box} and Proposition~\ref{prop:UNmap}.

\begin{theorem}\label{thm:onbexplicit}
Let $K\ge2${, and} put $k=K-1$. Let $a\ge1$, and let $\alpha = [a_0; a_1, a_2, \ldots]$ be an
irrational number with $a_i \leq a$ for all $i \geq 1$. 
For each $N\ge1$, with $\lcmk\mid N$, let $U_N$ be the map given by Proposition \ref{prop:UNmap}.
Then:
\begin{enumerate}
\item Let $D_N = {\rm diag}(\frac{1}{N/k+1}, \frac{1}{N/(k-1)+1},\ldots, \frac{1}{N+1})$ and define for each $l \in \B_N$
  \begin{equation}
  \p_{l}(z):=
  \sum_{\substack{j\in\mathbb Z_{\ge0}^K\\j_1+\cdots+j_K=N}}
  \binom Nj^{1/2}z^j
  e\left(\alpha\sum_{s=1}^Kj_s^2
  +U_N(j_1,\ldots,j_k)\cdot D_N l\right),
  \label{eq:originalBasis}
  \end{equation}
 where
  $\binom Nj=N!/(j_1!\cdots j_K!)$.
  Then $\{\p_{l}:l\in\B_N\}$ is an orthonormal basis
  of $\PolySpace_N${.}
\item Assuming the estimate in Theorem~\ref{thm:box},
  \[
  \sup_{\substack{N\ge1\\ \lcmk\mid N}}
  \ \sup_{l\in\B_N}
  \|\p_{l}\|_{L^\infty(\Sphere^{2K-1})}\ll_{K,a}1{.}
  \]
\item The divisibility condition can be
  removed: assuming Theorem~\ref{thm:box}, then, for every $N\ge0$, the space $\PolySpace_N$ admits an orthonormal basis
  $(\psi_\nu)_{1\leq \nu\le\dim\PolySpace_N}$ satisfying
  \[
    \sup_{1\leq \nu\le\dim\PolySpace_N}
    \|\psi_\nu\|_{L^\infty(\Sphere^{2K-1})}\ll_{K,a}1.
  \]
\end{enumerate}
\end{theorem} 

{The paper is organized as follows. In Section~\ref{part:one}, we deduce orthonormality from the bijection $U_N$ and the orthogonality of the Fourier characters on $\B_N$. We then localize the multinomial weights to boxes and use Fourier inversion to reduce uniform boundedness to the sectional estimate in Theorem~\ref{thm:box}. An induction on the dimension, combined with Olevskii's absorption lemma, removes the divisibility condition, thus, we only need the equidecompositions of Proposition~\ref{prop:UNmap} for degrees satisfying $\lcmk\mid N$. The next two sections establish the exponential sum estimate for balanced pieces in Theorem~\ref{thm:general} and the decomposition into such pieces in Theorem~\ref{thm:decomp}, which together yield Theorem~\ref{thm:box}. Finally, Section~\ref{sec:explicitEquidecomposition} proves Proposition~\ref{prop:UNmap} in two ways: through unimodular triangulations and through an explicit recursive construction. These ingredients complete the proof of Theorem~\ref{thm:onbexplicit}, from which Theorem~\ref{thm:onb} follows immediately.}

\section{Reduction to exponential sums: Proof of Theorem \ref{thm:onbexplicit} assuming Theorem \ref{thm:box} and Proposition \ref{prop:UNmap}}\label{part:one}
Recall that for $z=(z_1,z_2,\ldots,z_K) \in \cp^K$ and $j\in \z_{\geq 0}^K$ we write $z^j=z_1^{j_1}\cdots z_K^{j_K}$. We also set $k=K-1$. 
For $z \in \Sphere^{2K - 1}$ we {put} $z_s = \sqrt{r_s} e(\theta_s)$, for $s=1,{\ldots},K$, where $e(z)=e^{2\pi i z}$,  $\theta_s\in \RR$ and $r_s \geq 0$ so that $\sum_{s=1}^K r_s = 1$. In order to estimate uniformly
\[
\|\p\|_{L^\infty(\mathbb B_K)} = \|\p\|_{L^\infty(\Sphere^{2K-1})} = \sup_{\substack{r_1+\cdots+r_{K}=1 \\ r_s\geq 0, \ \theta_s \in \RR}} |\p( \sqrt{r_1} e(\theta_1),\ldots,\sqrt{r_K} e(\theta_K))|
\]
for $\p \in {\PolySpace_N}$, it will be useful to define the following auxiliary function.  For $\lambda>0$, define the square root of the Poisson
weight, extended by zero to the negative integers, by
\[
q_\lambda(n)=
\begin{cases}
\displaystyle e^{-\lambda/2}\frac{\lambda^{n/2}}{\sqrt{n!}},&n\geq0,\\[5pt]
0,&n<0.
\end{cases}
\]
Writing $r=(r_1,{\ldots},r_K)$, {for $N\ge1$} define for every $n\in\z^K$ the normalized multinomial weights
\[
a_{N,r}(n):=
\begin{cases}
\binom Nn^{1/2}\prod_{s=1}^Kr_s^{n_s/2},&n_1,{\ldots},n_K\ge0 \text{ and  } n_1+\cdots+n_K=N,\\[4pt]
0,&\text{otherwise}.
\end{cases}
\]
It is straightforward to see that when $a_{N,r}(n) \neq 0$ we have
\begin{align}
a_{N,r}(n)=c_N\prod_{s=1}^Kq_{Nr_s}(n_s)
\quad \text{with} \quad
c_N:=e^{N/2}\frac{\sqrt{N!}}{N^{N/2}}\asymp N^{1/4},
\label{eq:poissonFactorizationMultinomial}
\end{align}
where the estimate for $c_N$ follows from Stirling's \cite[Section~5.11]{DLMF}. Note that any $\p \in {\PolySpace_N}$, in the new coordinates $z_s = \sqrt{r_s} e(\theta_s)$, can be {expressed} as
\[
\p(z_1,\ldots,z_K) = \sum_{\substack{n \in \Z_{\geq 0}^K  \\  n_1+\cdots+n_K=N}} c(n) a_{N,r}(n) e(\theta \cdot n),
\]
for some coefficients $c(n)$, where $\theta = (\theta_1,\ldots,\theta_K)$. In particular, $\|\p\|^2 = \binom{N+K-1}{K-1}^{-1} \sum  |c(n)|^2$.  Put
\[
H_\lambda:=(1+\lambda)^{1/2}  \quad \text{and} \quad
\qquad \rho_\lambda(n):=\frac{|n-\lambda|}{H_\lambda}.
\]
{Also set} $\Delta b(j)=b(j+1)-b(j)$, for a sequence $b:\Z\to \cp$.  The next lemma produces the necessary bounds for the components of $a_{N,r}(n)$.

\begin{lemma}\label{lem:poisson1d}
We have, for all $n \in \mathbb{Z}$,
\begin{align}
q_\lambda(n)\ll H_\lambda^{-1/2}e^{-\rho_\lambda(n)/4} \quad \text{and} \quad
|\Delta q_\lambda(n)|\ll H_\lambda^{-3/2}e^{-\rho_\lambda(n)/4}.
\label{eq:poissonPointwiseBound}
\end{align}
\end{lemma}

\begin{proof}
Stirling's inequalities give
\begin{equation} \label{eq:lemma_first}
q_\lambda(n)^2\ll(1+n)^{-1/2}e^{-F_\lambda(n)},
\end{equation}
where $F_\lambda(n):=n\log(n/\lambda)-n+\lambda= {\int_{\la}^n (n-t)\frac{\dd t}{t}}$. Separating into cases, $n\geq \la \geq 0$ and $\la > n\geq 0$, we can bound the integral from below by
\begin{equation} \label{eq:lemma_second}
F_\lambda(n)\ge\frac{(n-\lambda)^2}{2(\lambda+|n-\lambda|)}
\ge\frac{\rho_\lambda(n)^2}{2(1+\rho_\lambda(n))}\ge\frac{\rho_\lambda(n)-1}{2}.
\end{equation}
The elementary inequality $\frac{1}{a} \leq \frac{1}{b} + |\frac{a^2}{b^2}-1|$, for $a,b\geq 1$, applied to $a=(n+1)^{1/2}$ and $b=H_\lambda$ implies
\begin{equation} \label{eq:lemma_bound_ineq}
(1+n)^{-1/2}\le H_\lambda^{-1}(1+\rho_\lambda(n))
\ll H_\lambda^{-1}e^{\rho_\lambda(n)/4}.
\end{equation}
To bound $\Delta q_{\lambda}$ we note that for $n \geq 0$ 
\[
\Delta q_\lambda(n)
=q_\lambda(n)\frac{\lambda-n-1}
{\sqrt{n+1}(\sqrt\lambda+\sqrt{n+1})}.
\]
Since $\sqrt\lambda+\sqrt{n+1}\ge H_\lambda$, the quotient satisfies
\[
\left|\frac{\lambda-n-1}
{\sqrt{n+1}(\sqrt\lambda+\sqrt{n+1})}\right|
\le\frac{1+\rho_\lambda(n)}{\sqrt{n+1}}
\le\frac{(1+\rho_\lambda(n))^2}{H_\lambda},
\]
where in the last inequality we reused the first inequality in \eqref{eq:lemma_bound_ineq}. Combining this with \eqref{eq:lemma_first} and \eqref{eq:lemma_second}, we get 
$$
| \Delta q_{\lambda}(n) | \ll H_{\lambda}^{-3/2} (1 + \rho_{\lambda}(n))^{5/2} e^{-\rho_{\lambda}(n) / 2} \ll H_{\lambda}^{-3/2} e^{-\rho_{\lambda}(n) / 4}.
$$
Also $|\Delta q_{\lambda}(-1)| = e^{-\lambda / 2}$ and $|\Delta q_{\lambda}(- n)| = 0$ for $n > 1$. Therefore the above inequality also holds for negative $n$.
\end{proof}

Fix the hat function $\eta(x)=(1-|x|)_+$.
It is supported in $[-1,1]$, bounded by one, $1$-Lipschitz and {satisfies}
\begin{align}
\sum_{n\in\z}\eta(x-n)=1,
\label{eq:shiftFreePartitionUnity}
\end{align}
for every $x\in\RR$. For a finitely supported $b:\z\to\cp$, write
\[
\ft b(u)=\sum_{j\in\z}b(j)e(-uj) \quad \text{for} \quad u\in \RR.
\]

\begin{lemma}\label{lem:fouriermass}
For $m\in\z$ put
\begin{align}
q_{\lambda,m}(n):=q_\lambda(n)\eta(n/H_{\la}-m).
\label{eq:localizedMultinomialSequence}
\end{align}
Then $q_\lambda=\sum_mq_{\lambda,m}$, each $q_{\lambda,m}$ is
supported on $\{n:\ |n/H_\lambda-m|<1\}$ and
\begin{align}
\sum_{m\in\z}\int_0^1|\ft{q_{\lambda,m}}(u)|\,\dd u\ \ll\ H_\lambda^{-1/2},
\label{eq:multinomialTotalFourierMass}
\end{align}
where the implied constant {is} independent of $\la >0$.
\end{lemma}

\begin{proof}
Write $H=H_\lambda$. The first two assertions follow from
\eqref{eq:shiftFreePartitionUnity} and the support of $\eta$. For any finitely supported $b:\z\to\cp$, Cauchy--Schwarz implies
\[
\int_0^1|\ft b(u)|\,\dd u
\le\Bigl(\int_0^1\frac{\dd u}{1+H^2|e(u)-1|^2}\Bigr)^{1/2}
\Bigl(\int_0^1\bigl(1+H^2|e(u)-1|^2\bigr)|\ft b(u)|^2\,\dd u\Bigr)^{1/2}.
\]
Since $|e(u)-1|=2|\sin\pi u|\ge4\min\{u,1-u\}$, the first factor is
$\ll H^{-1/2}$. As $\ft{\Delta b}(u)=(e(u)-1)\ft b(u)$, Parseval
evaluates the square of the second factor as
$\|b\|_{\ell^2}^2+H^2\|\Delta b\|_{\ell^2}^2$. Hence
\begin{align}
\int_0^1|\ft b(u)|\,\dd u
\ \ll\ H^{-1/2}\bigl(\|b\|_{\ell^2(\z)}+H\|\Delta b\|_{\ell^2(\z)}\bigr).
\label{eq:multinomialDiscreteSobolev}
\end{align}
Since $\eta$ is $1$-{Lipschitz} we have
\[
|\Delta q_{\lambda,m}(n)|
\le |\Delta q_\lambda(n)|+H^{-1}q_\lambda(n)
\ll H^{-3/2}\exp\Bigl(-  \frac{|n-\lambda|}{4H}\Bigr),
\]
using the previous lemma in the last inequality.
Both $q_{\lambda,m}$ and $\Delta q_{\lambda,m}(n)$ are supported on
integers $n$ satisfying $|n/H-m|\le2$. Thus Lemma~\ref{lem:poisson1d} gives
\[
\|q_{\lambda,m}\|_2+H\|\Delta q_{\lambda,m}\|_2
\ll \exp\Bigl(- \frac{1}{4} \cdot \Bigl|m-\frac\lambda H\Bigr|\Bigr),
\]
which by \eqref{eq:multinomialDiscreteSobolev} implies
\[
\sum_{m\in\Z}\int_0^1|\widehat{q_{\lambda,m}}(u)|\,\dd u
\ll H^{-1/2}\sum_{m\in\Z}
\exp\Bigl(-\frac{1}{4} \cdot \Bigl|m-\frac\lambda H\Bigr|\Bigr)
\ll H^{-1/2},
\]
uniformly.
\end{proof}

\subsection{Uniform boundedness for a more general shifted basis}
Throughout this section we will assume Theorem \ref{thm:box} and Proposition \ref{prop:UNmap}. 
We start by recalling the collection of functions we have already defined in Theorem \ref{thm:onbexplicit}.  
Let {$a\geq1$, and let $\alpha = [a_0; a_1, a_2, \ldots]$ be an}
irrational number with $a_i \leq a$ for all $i \geq 1$. Let $\cS_{N}:=\left\{{n}\in \Z_{\ge0}^k:
  n_1+\cdots+n_k\le N\right\}$, $\B_{N}:=\prod_{s=1}^k \{0,1,\ldots, \frac{N}{k+1-s}\}$ and $U_N:\cS_N \to \B_N$ be the map given by Proposition \ref{prop:UNmap}. For  $\lcmk\mid N$  write
   \begin{equation}
  \p_{l}(z):=
  \sum_{\substack{n\in\mathbb Z_{\ge0}^K\\n_1+\cdots+n_K=N}}
  {\binom Nn^{1/2}z^n}
  e\left(\alpha\sum_{s=1}^Kn_s^2
  +U_N(n_1,\ldots,n_k)\cdot D_Nl\right),
  \label{def:oribasis}
  \end{equation}
 for each $l\in \B_N$, where $D_N = {\rm diag}(\frac{1}{N/k+1}, \frac{1}{N/(k - 1)+ 1},\ldots, \frac{1}{N+1})$. The basis depends on \(\alpha\), but we suppress this dependence in the notation. The bounds we are going to derive are uniform over all \(\alpha\) with bounded partial quotients.

\begin{definition}[Shifted basis]\label{def:shiftedbasis}
Let $N\geq 0$ and set
\begin{align}
b:=N-\lcmk\Bigl\lfloor\frac N{\lcmk}\Bigr\rfloor \quad \text{and} \quad N'=N-b.
\end{align}
Also write
\begin{align}
D:=\binom{N+k}{k}=\dim\PolySpace_N\quad \text{and}\quad
D':=\binom{N'+k}{k}=\dim\PolySpace_{N'}.
\label{eq:dimensions}
\end{align}
For $l\in\B_{N'}$ put
\begin{align}
\p^b_{l}(z)
=\Bigl(\frac{D}{D'}\Bigr)^{1/2}
\sum_{\substack{n\in\z_{\ge0}^K,\\ n_1+\cdots+n_K=N\\ n_K\ge b}}
\binom Nn^{1/2}z^n\,
e\Bigl(\al\sum_{s=1}^Kn_s^2
+U_{N'}(n_1,\ldots,n_k)\cdot D_{N'}l\Bigr).
\label{eq:shiftedBasis}
\end{align}
\end{definition}
Note that when $b=0$ the collection $\{\p^b_{l}: l\in \B_{N'}\}$ is exactly the one defined in \eqref{def:oribasis}. We consider the more general basis elements with a shift by $b$ for the purpose of eventually removing the restriction that $\lcmk | N$.  

\begin{proposition}\label{prop:shiftedONS}
The functions $\{\p^b_{l}:l\in\B_{N'}\}$ form an orthonormal basis of the 
subspace spanned by the monomials $z^n \in {\PolySpace_N}$ with
$n_K\ge b$.
\end{proposition}

\begin{proof}
By \eqref{def:innerprod} we have
\begin{align*}
\langle\p^b_{l},\p^b_{l'}\rangle
=\frac1{D'}
\sum_{x\in\cS_{N'}}
e\bigl(U_{N'}(x)\cdot D_{N'}(l-l')\bigr)
& =\frac1{D'}\sum_{w\in\B_{N'}}
e\bigl(w\cdot D_{N'}(l-l')\bigr) \\
& = \prod_{s=1}^k \Biggl(  \frac{1}{{N'}/s+1}\sum_{v=0}^{{N'}/s} e\left(\frac{v (l_{k+1-s}-l'_{k+1-s})}{{N'}/s+1} \right)\Biggr)
\\ & =\delta_{l,l'}.
\end{align*}
Since the stated span has dimension $D'=\# \B_{N'}$, the proposition follows.
\end{proof}

Write $z_s=r_s^{1/2}e(\theta_s)$, so that $r_s\geq0$, $\sum_{s=1}^Kr_s=1$,
and $\theta_s \in \RR$. Put $r=(r_1,\ldots,r_K)$ and $\theta=(\theta_1,\ldots,\theta_K)$. 
Set $H_s:=H_{Nr_s}=(1+Nr_s)^{1/2}$ and
\begin{align}
H_q=\max_{1\le s\le K} H_s\asymp_K {(1+N)^{1/2}}.
\label{eq:largestScale}
\end{align}

We now give a brief informal overview {of} how we obtain the uniform boundedness of $\p^0_{l}$ from Theorem \ref{thm:box} and Proposition \ref{prop:UNmap}. On the hyperplane $n_1+\cdots+n_K=N$ {the} normalized multinomial weight
$a_{N,r}(n)$ is a product of one-dimensional
functions of the coordinates $n_s$ and each
factor is concentrated at scale $H_s=(1+Nr_s)^{1/2}$ around $Nr_s$.
We localize each factor at its scale and Fourier-invert so the total
Fourier mass is $\ll\prod_sH_s^{-1/2}$ (Lemma~\ref{lem:fouriermass}).
The exponential sum over the resulting box, cut by the linear
equation $n_1 + \cdots + n_K = N$, is then $\ll\prod_{s\ne q}H_s^{1/2}$ by
Theorem~\ref{thm:box}, where $r_q=\max_s r_s$. The normalization
constant $c_N$ then gives
$c_NH_q^{-1/2}\ll_K1$, since
$c_N\asymp N^{1/4}$ and $H_q\asymp_K N^{1/2}$.

\begin{proposition}\label{prop:shiftedbound}
Let $a\ge1$ and let $\alpha=[a_0;a_1,a_2,\ldots]$ be an irrational number
with $a_i\le a$ for all $i\ge1$.  Let $N\ge0$, with $b \in [0, \lcmk)$ the residue of $N$ modulo $\lcmk$ and $N'=N-b$, and let $\p^b_{l}(z)$ be defined as in \eqref{eq:shiftedBasis}. Then there is a constant $A=A(K,a)$ such that
\[
  \sup_{l\in\B_{N'}}\ \sup_{z\in\Sphere^{2K-1}}
  \bigl|\p^b_{l}(z)\bigr|\ \le\ A.
\]
\end{proposition}

\begin{proof}{
We can assume $N\geq1$. {By continuity we may also assume $r_s>0$ for every $s$.}
First note that
\[
\frac{D}{D'}=\prod_{s=1}^k\left(1+\frac{b}{N'+s}\right)
\leq\prod_{s=1}^k\left(1+\frac{\lcmk}{s}\right)\ll_K1.
\]
By Proposition \ref{prop:UNmap}, there are
finite sets $L_i\subset\Z^k\times\Z$ and affine maps
$A_ix+a_i\in\aff_k(\Z)$, for $1\leq i\leq n_0$, such that
\[
\cS_{N'}=\bigsqcup_{i=1}^{n_0}\cP(L_i) \quad \text{and} \quad
U_{N'}(x)=A_ix+a_i\quad \text{for} \quad x\in\cP(L_i),
\]
where $n_0\ll_K1$, $\#L_i\ll_K1$ and $|v|_1\ll_K1$ for every $(v,d)\in L_i$. Let $\pi : \mathbb{Z}^k \rightarrow \mathbb{Z}^K$ be defined by $\pi((v_1, \ldots, v_k)) = (v_1, \ldots, v_k, 0)$.
Define the lifted set $L_i'=\{(\pi(v),d) : (v,d)\in L_i\}\cup {\{({\bf 1},-N),(-{\bf 1},N)\}}$, where ${\bf 1} = (1,\ldots, 1)\in {\Z^K}$ so that
\begin{align*}
\{ n \in \Z^K : n_1+\cdots+n_K=N,\ n_K\ge b\}
 & = \{ n \in \Z^K : (n_1,{\ldots},n_k)\in\cS_{N'},\ n_1+\cdots+n_K=N\} \\
 & = \bigsqcup_{i=1}^{n_0}\cP(L_i').
\end{align*}
Therefore,
\begin{align}
\p^b_{l}(z)=\Bigl(\frac D{D'}\Bigr)^{1/2}
\sum_{i=1}^{n_0}e\bigl(a_i\cdot D_{N'}l\bigr)
\sum_{{n\in\cP(L_i')}}
a_{N,r}(n)\,e\Bigl(\al\sum_{s=1}^Kn_s^2+\xi_i\cdot n\Bigr),
\label{eq:basisSplitIntoPieces}
\end{align}
where $\xi_i:=\theta+\pi(A_i^{\mathsf T}D_{N'}l)\in\RR^K$. By \eqref{eq:poissonFactorizationMultinomial} we have
$a_{N,r}(n)=c_N\prod_{s=1}^Kq_{Nr_s}(n_s)$
{for $n\in\cP(L_i')$}. Fourier inversion then gives
\[
a_{N,r}(n)=\sum_{m\in \Z^K} c_N  \int_{[0,1]^K}\prod_{s=1}^K \ft{q_{Nr_s,m_s}}(u_s) e(n\cdot u) \dd u,
\]
where $u=(u_1,\ldots,u_K)$. Since $\prod_{s=1}^K {q_{Nr_s,m_s}}(n_s)$ is supported in the box $J_m=\{n\in \Z^K : |n_s/H_s-m_s| \leq 1\ {\text{for }1\le s\le K}\}$, for $m\in \Z^K$, we commute the sums in \eqref{eq:basisSplitIntoPieces} {to obtain}
\begin{align}
|\p^b_{l}(z)|\ll_K c_N\sum_{i=1}^{n_0}\sum_{m\in\Z^K}
\int_{[0,1]^K}\prod_{s=1}^K|\ft{q_{Nr_s,m_s}}(u_s)|\,
\Bigl|\sum_{n\in \cP(L_i') \cap J_m}
e\Bigl(\al\sum_{s=1}^Kn_s^2+(\xi_i+u)\cdot n\Bigr)\Bigr|\,\dd u.
\end{align}
The sectional bound of Theorem \ref{thm:box} then applies in dimension $K$,
with all $c_j=1$ and $w_0=\mathbf1$, producing
\begin{align}
\Bigl|\sum_{n\in \cP(L_i') \cap J_m}
e\Bigl(\al\sum_{s=1}^Kn_s^2+(\xi_i+u)\cdot n\Bigr)\Bigr|
\ll_{K,a}\sect(\mathbf1;J_m)^{1/2}
\leq\Bigl(\prod_{s\ne q}\#J_{m,s}\Bigr)^{1/2}
\ll_K\prod_{s\ne q}H_s^{1/2},
\label{eq:sectionalVolumeOfTile}
\end{align}
because $\len(\mathbf1;J_m)\geq\#J_{m,q}$. Here $H_q=\max_{1\leq s\leq K} H_s \asymp_K (1+N)^{1/2}$. 
We deduce
\begin{align*}
|\p^b_{l}(z)|
&\ll_{K,a}c_N\prod_{s\ne q}H_s^{1/2}
\prod_{s=1}^K\Bigl(\sum_{m_s\in\Z}\int_0^1
|\ft{q_{Nr_s,m_s}}(u_s)|\,\dd u_s\Bigr)\\
&\ll_{K,a}c_N\prod_{s\ne q}H_s^{1/2}\prod_{s=1}^KH_s^{-1/2}
=c_NH_q^{-1/2}\ll_K1,
\end{align*}
where we used Lemma \ref{lem:fouriermass} and that $c_N \asymp_K {N^{1/4}}$.
This finishes the proof.
}
\end{proof}

\subsection{Removing the divisibility condition}\label{sec:alldegrees}
Let $N\geq 0$ and put
$$
\H_{K,N}^{(i)} = \text{span}\{z^n : n\in \Z_{\geq 0}^K, \, n_1+\cdots+n_K=N, \,  n_K\geq i\}
$$
and set $\H_{K,N}:=\H_{K,N}^{(0)}{=\PolySpace_N}$. Let $b \in [0, \lcmk)$ {be} the residue of $N$ modulo $\lcmk$ and $N'=N-b$. We start by noticing that
\begin{align}\label{eq:Hfacesdecomp}
\H_{K,N} =  \H_{K,N}^{(b)} \bigoplus \bigg(\bigoplus_{0\leq i<b} z_K^i \H_{K-1,N-i} \bigg),
\end{align}
where this is an orthogonal decomposition such that each subspace $z_K^i \H_{K-1,N-i}$ is isometric to ${\H_{K-1,N-i}}$ (Lemma \ref{lem:facenorm} below). In the subspace $\H_{K,N}^{(b)}$ we have already constructed {a} uniformly bounded orthonormal basis by Propositions \ref{prop:shiftedONS} and \ref{prop:shiftedbound}. The idea then is to construct the final desired uniformly bounded orthonormal basis for all $N$ by a recursive method, using this decomposition,  and attaching to each $z_K^i \H_{K-1,N-i}$ the basis previously constructed in the lower dimension. {A priori} these will not be uniformly bounded, but they will have the property that the exceptional set has at most ${O_K(N^{K-2})}$ functions, all uniformly bounded by $O_K(\sqrt{N})$. In this situation one can appeal to Olevskii's absorption Lemma \ref{lem:absorption} (below), which applies a Haar rotation to redistribute the large functions among many other small basis {functions} to obtain a final uniformly bounded orthonormal set.

\begin{lemma}[The faces]\label{lem:facenorm}
{Let $K\ge2$, $N\ge1$ and $0\le i\le N$ be integers.} Put 
\begin{align}
c_{N,i}^2
:=\frac{\binom{N+K-1}{K-1}\binom Ni}
{\binom{N-i+K-2}{K-2}}.
\label{eq:facenorm}
\end{align}
Then  $c_{N,i}\asymp_{K,i} N^{(i+1)/2}$.
Moreover:
\begin{enumerate}
\item The map $\psi\mapsto c_{N,i}\,z_K^{\,i}\psi$ is a linear isometry
from ${\H_{K-1,N-i}} \subset  L^2(\Sphere^{2K-3})$ onto
$z_K^i \H_{K-1,N-i} \subset L^2(\Sphere^{2K-1})${.}
\item If ${\H_{K-1,N-i}}$ has an orthonormal basis
$\{\psi_1,{\ldots},\psi_{{\dim\H_{K-1,N-i}}}\}$ with
$\|\psi_\nu\|_{L^\infty(\Sphere^{2K-3})}\le C$, then the functions
$c_{N,i}z_K^{\,i}\psi_\nu$ form an orthonormal basis of $z_K^i \H_{K-1,N-i}$ satisfying
\begin{align}
\bigl\|c_{N,i}z_K^{\,i}\psi_\nu\bigr\|_{L^\infty(\Sphere^{2K-1})}
 \ll_{K,i}\ C\,N^{1/2}.
\label{eq:facesup}
\end{align}
\end{enumerate}
\end{lemma}

\begin{proof}
Recall that
for every monomial $\psi(z_1,{\ldots},z_{K-1})$ of degree $N-i$ with individual degrees $m_1, \ldots, m_{K - 1}$ we have
$$
\| z_K^i \psi \|^2_{L^2(\Sphere^{2K - 1})} = \binom{N + K - 1}{K - 1}^{-1} \cdot \binom{N}{m_1, \ldots, m_{K - 1}, i}^{-1}
$$
and
$$
\| \psi \|^2_{L^2(\Sphere^{2K - 3})} = \binom{N - i + K - 2}{K - 2}^{-1} \binom{N - i}{m_1, \ldots, m_{K - 1}}^{-1}.
$$
Recall also the identity
$$
\binom{N}{m_1, \ldots, m_{K - 1} , i} = \binom{N}{i} \cdot \binom{N -i }{m_1, \ldots, m_{K - 1}}.
$$
Therefore,
$$
\| c_{N,i} z^i_K \psi \|^2_{L^2(\Sphere^{2K - 1})} = c_{N,i}^2 \| z^i_K \psi \|_{L^2(\Sphere^{2K - 1})}^2 = \| \psi \|_{L^2(\Sphere^{2K - 3})}^2.
$$
Distinct monomials remain orthogonal, so the same norm identity holds
for every $\psi\in{\H_{K-1,N-i}}$. This proves the isometry. 
Put $\tau=|z_K|^2$. 
Since $|z_1|^2 + \cdots + |z_K|^2 = 1$ we can write 
$(z_1, \ldots, z_{K - 1}) = (1 - \tau)^{1/2} w$ with $w \in \mathbb{S}^{2K - 3}$.
Since $\psi$ is homogeneous of degree $N - i$, 
$$|\psi(z_1, \ldots, z_{K - 1})| = (1 - \tau)^{(N - i) / 2} |\psi(w)| \leq C (1 - \tau)^{(N - i) / 2}.$$ 
Moreover it is easy to see that $c_{N,i} \ll_{K,i} N^{(i + 1) / 2}$. It follows that
\[
| c_{N,i} z_K^i \psi_\nu(z_1,{\ldots},z_{K-1})|\ll_{K,i}  C N^{(i+1)/2} \tau^{i/2}(1-\tau)^{(N-i)/2}.
\]
The maximum of $\tau^i(1-\tau)^{N-i}$ occurs at $\tau=i/N$ and this produces the final bound  $\ll_{K,i} CN^{1/2}$.
\end{proof}
The term $N^{1/2}$ in \eqref{eq:facesup} is sharp. However, for
$0\le i<b$ the orthonormal bases $\{c_{N,i}z_K^i\psi_\nu\}$ of
$z_K^i\H_{K-1,N-i}$ given in Lemma~\ref{lem:facenorm} consist of
\[
\sum_{0\le i<b}\dim\H_{K-1,N-i}
=\sum_{0\le i<b}\binom{N-i+K-2}{K-2}\ll_K N^{K-2}
\]
functions, whereas $\dim\H_{K,N}=\binom{N+K-1}{K-1}\asymp_K N^{K-1}$.
Because there are only a few of them, an orthogonal change of basis that mixes them with the uniformly bounded functions $\p^b_{l}$ produces an orthonormal basis of $\H_{K,N}$ that is uniformly bounded. This is Olevskii's
absorption scheme \cite{Olv1}, used by Bourgain \cite[Section~3]{Bour2}
for $K=3$ and $N$ odd. We include a self-contained proof. 

\begin{lemma}[Absorption Lemma]\label{lem:absorption}
Let $D \geq 2$. Let $X$ be a set. Let $V$ be a $D$-dimensional complex vector space
of functions $f: X\to\mathbb C$.
Let $\p_1,{\ldots},\p_D$ be an orthonormal
basis of $V$ with respect to an inner product $\langle \cdot, \cdot \rangle : V \times V \rightarrow \mathbb{C}$. Let $A,B\ge0$ and let $D_2$ be an integer with
$1\le D_2\le D/2$. Suppose
\[
\sup_{x \in X} |\p_j(x)| \leq B \quad(1\le j\le D_2)
{\quad \text{and} \quad}
\sup_{x \in X} |\p_j(x)| \le A \quad(D_2<j\le D).
\]
Then $V$ has an orthonormal basis $\psi_1,{\ldots},\psi_D$ with respect to 
the inner product $\langle \cdot, \cdot \rangle$ such that
\[
\max_{1\le s\le D} \sup_{x \in X} |\psi_s(x)|
\le 10A+10B\sqrt{D_2/D}.
\]
\end{lemma}

\begin{proof}
Let $R_1=(1)$ and, for $n\ge1$, define
\begin{align}
R_{2^n}:=\frac1{\sqrt2}
\begin{pmatrix}
R_{2^{n-1}}&I_{2^{n-1}}\\
R_{2^{n-1}}&-I_{2^{n-1}}
\end{pmatrix},
\label{eq:haarrecursion}
\end{align}
where $I_m$ is the $m\times m$ identity matrix.
It can be seen by induction that this is an orthogonal matrix. 
Indeed if $R_{2^{n - 1}} R_{2^{n - 1}}^T = I_{2^{n - 1}}$ then 
$$
R_{2^n} R_{2^n}^{T} = \frac{1}{2} \begin{pmatrix}
R_{2^{n - 1}} R_{2^{n - 1}}^T + I_{2^{n - 1}} & R_{2^{n - 1}} R_{2^{n - 1}}^T - I_{2^{n - 1}} \\ 
R_{2^{n - 1}} R_{2^{n - 1}}^T - I_{2^{n - 1}} & R_{2^{n - 1}} R_{2^{n - 1}}^T + I_{2^{n - 1}}
\end{pmatrix}  = I_{2^n}.
$$ 
Let $(r_{s,1},{\ldots},r_{s,2^n})$ be a row of $R_{2^n}$.
It is easy to see from the inductive definition that $r_{s,1}=2^{-n/2}$. 
Moreover, for each $0\le\ell<n$,
exactly one entry with $2^\ell<j\le2^{\ell+1}$ is nonzero,
of modulus $2^{(\ell-n)/2}$. This also follows by induction from
\eqref{eq:haarrecursion}: the first half of a row is a row of
$R_{2^{n-1}}$ divided by $\sqrt2$, and the second half has
a single nonzero entry of modulus $2^{-1/2}$. Consequently,
if $2\le d\le2^n$ and $2^L<d\le2^{L+1}$, then
\begin{align}
\sum_{j\le d}|r_{s,j}|
\le2^{-n/2}\left(1+\sum_{\ell=0}^L2^{\ell/2}\right)
\le5\sqrt d\,2^{-n/2}.
\label{eq:haartail}
\end{align}
The same bound holds for $d=1$. Now let $n=\lfloor\log D / \log 2\rfloor$, so that
$2^n\le D<2^{n+1}$, in particular, $D_2\le D/2<2^n$. Put
\[
Q = (Q_{s,j})_{\substack{1 \leq s \leq D \\ 1 \leq j \leq D}} := \begin{pmatrix} R_{2^n} & 0 \\ 0 & I_{D - 2^n} \end{pmatrix}
  \quad \text{and} \quad
\psi_s:=\sum_{j=1}^D Q_{s,j}\p_j.
\]
Since $Q Q^T = I_D$, the $\psi_i$ are orthonormal:
$$
\langle \psi_i, \psi_j \rangle = \sum_{1 \leq s, t \leq D} Q_{i, s} Q_{j, t} \langle \p_s, \p_t \rangle = \sum_{s = 1}^{D} Q_{i,s} Q_{j, s} = \delta_{i j}.
$$
Applying \eqref{eq:haartail} with $d=D_2$
and $d=2^n$ gives for $s \leq 2^n$
\begin{align*}
\sup_{x \in X} |\psi_s(x)|
\le B\sum_{j\le D_2}|Q_{s,j}|
   +A\sum_{D_2<j\le2^n}|Q_{s,j}|
\le5B\sqrt{D_2}\,2^{-n/2}+5A
\le10A+10B\sqrt{D_2/D}.
\end{align*}
For $s>2^n$ we have $\psi_s=\p_s$ and $|\psi_s(x)| \leq A$ for all $x \in X$, which finishes the proof.
\end{proof}

\subsection{Proof of Theorem \ref{thm:onbexplicit} assuming Theorem \ref{thm:box} and Proposition \ref{prop:UNmap}}
\begin{proof}
(1) and (2) follow from Propositions \ref{prop:shiftedONS} and ~\ref{prop:shiftedbound} applied to $b=0$. 
Fix $a$ and $\alpha$ as in the theorem. We will prove (3) by induction on $K$.  

\emph{Base case.} For $K=2$ we have $\lcmk = 1$. Therefore we can simply apply Propositions \ref{prop:shiftedONS} and ~\ref{prop:shiftedbound}, which already prove that the orthonormal set $\Psi_{2,N}:=\{\p^0_{l}: l=0,1,{\ldots}, N\}$, given by Definition ~\ref{def:shiftedbasis}, is uniformly bounded and spans the whole space $\H_{2,N}$.

\emph{Inductive step.} Let ${K\ge3}$ and assume that for all $N \geq 0$
there exists an orthonormal basis $\Psi_{K - 1, N}$ of $\H_{K-1,N}$
such that for every $\psi \in \Psi_{K - 1, N}$ we have $\| \psi \|_{L^{\infty}(\Sphere^{2K - 3})} \leq C_{K - 1}$.
We assume without loss of generality that $N \geq 2 (K - 1) \lcmk$, as for the smaller $N$ any basis would do.
Let $b \in [0, \lcmk)$ be the residue of $N$ modulo $\lcmk$ and $N'=N-b$.
Recall the orthogonal decomposition  \eqref{eq:Hfacesdecomp},
$$\H_{K,N} =  \H_{K,N}^{(b)} \bigoplus \left(\bigoplus_{0\leq i<b} z_K^i \H_{K-1,N-i} \right).$$
By Propositions \ref{prop:shiftedONS} and \ref{prop:shiftedbound} the subspace $\H_{K,N}^{(b)}$ admits the orthonormal basis $\{\p^b_{l}:l\in\B_{N'}\}$ with
$$
\| \p^b_{l} \|_{L^{\infty}(\Sphere^{2K - 1})} \leq A := A(K,a).
$$
If $b=0$ this is already the required basis of $\H_{K,N}$, so assume $b\ge1$.
By Lemma~\ref{lem:facenorm}(1) the following is an orthonormal basis for $\H_{K,N}$:
\[
\Phi_{K,N}= \bigcup_{0\leq i  < b} \Bigl(c_{N,i}\,z_K^{\,i} \Psi_{K-1,N-i}\Bigr) \cup \{\p^b_{l} : l\in \B_{N'}\},
\]
where $c_{N,i}$ are the constants in \eqref{eq:facenorm}.
By Lemma~\ref{lem:facenorm}(2) and the inductive hypothesis, for each $0\le i<b$ and each $\psi \in \Psi_{K - 1, N - i}$ we have
$$
\| c_{N,i} z_K^i \psi \|_{L^{\infty}(\Sphere^{2K-1})} \leq B
$$
with $B \ll_{K,a} N^{1/2}$.
Put
$$
D_2:=\dim \left(\bigoplus_{0\leq i<b} z_K^i \H_{K-1,N-i}\right) =\sum_{i=0}^{b-1}\binom{N-i+k-1}{k-1} \ll_K N^{K - 2}
$$
and
$$
D:=\dim \H_{K,N}=\binom{N+k}{k} \asymp_K N^{K - 1},
$$
so that $1\le D_2\le D/2$ (since $N \geq 2 (K - 1) \lcmk$).
Then Lemma~\ref{lem:absorption} gives an orthonormal basis $\Psi_{K,N}$ of $\H_{K,N}$ such that
for all $\psi \in \Psi_{K,N}$ we have
\[
 \|\psi\|_{L^\infty(\Sphere^{2K-1})} \leq 10A+10B\sqrt{D_2/D} \ll_{K,a}  1.
\] 
This completes the inductive step, and therefore finishes the proof of Theorem \ref{thm:onbexplicit}.  
\end{proof}

\section{Quadratic exponential sums: Proof of Theorem~\ref{thm:box} for balanced boxes}\label{sec:quadratic}

In this section we prove Theorem ~\ref{thm:box} under the extra assumption that the {summation set} is balanced (as defined below).
 
\begin{definition}\label{def:balanced}
Let $k\geq1$ and $R\ge1$. Let $I=\prod_{j=1}^kI_j$ be a box and
$L\subset\Z^k\times\Z$ a finite set of constraints. We say that the pair
$(L,I)$ is \emph{$R$-balanced} if
\[
 {\#L\le R \quad \text{and} \quad 0<|v|_1\le R \quad \text{for all }(v,b)\in L},
\]
and
\[
 \#I_i\le R\,\#I_j\ \ \text{ for all } i,j{.}
\]
The locations of the intervals and the constants $b$ are unrestricted. 
\end{definition} 
 
The following is the main result of this section.

\begin{theorem}\label{thm:general}
Let $k,R\geq1$ be integers. Let $(L,I)$ be an $R$-balanced pair.  Let $\alpha = [a_0; a_1, a_2, \ldots]$ 
be an irrational number with $a_i \leq R$ for all $i \geq 1$
and $c_1,{\ldots},c_k$ be positive integers with
$c_j\le R$.  Then
\[
\sup_{\theta\in\mathbb R^k} \Bigl|\sum_{n\in \cP(L) \cap I}e\Bigl(\alpha\sum_{j=1}^kc_jn_j^2
 +\theta\cdot n\Bigr)\Bigr|\ll_{k,R}\vol(I)^{1/2}.
\]
\end{theorem} 

The analytic core of Theorem \ref{thm:general} is the following proposition.
 
\begin{proposition}\label{prop:multiwall}
Let $A, C, V, k \geq 1$ and $r \geq 0$.
Let $\alpha = [a_0; a_1, a_2, \ldots]$ be an irrational number with
$a_i \leq A$ for all $i \geq 1$, and let $a/q$ be a convergent of
$\alpha$. Let $c_1,{\ldots},c_k$ be integers with $1\le c_j\le C$, let
$1 \leq P \leq q$ be an integer, let $w_1,{\ldots},w_r\in\Z^k\setminus\{0\}$ with
$|w_i|_1\le V$, and let $B_1,{\ldots},B_r$ be integers such that
\[
|w_i\cdot h-B_i|\le\frac q4
\qquad\text{whenever } |h_j|\le2P \text{ for all } j.
\]
Let $U$ be smooth and supported in $[-2,2]$, and let
$\beta\in\RR^k$. Then
\[
\Bigl|\sum_{h\in\mathbb Z^k}\ \prod_{i=1}^r
\operatorname{sgn}\bigl(2(w_i\cdot h-B_i)+1\bigr)
\prod_{j=1}^kU\Bigl(\frac{h_j}P\Bigr)\,
e\Bigl(\alpha\sum_{j=1}^kc_jh_j^2+\beta\cdot h\Bigr)\Bigr|
\ \ll\ q^{k/2},
\]
where the implied constant depends only on $A$, $k$, $r$, $V$, $C$ and
$\|U^{(m)}\|_{L^1}$ for $0\le m\le2$.

\end{proposition}

Section~\ref{sec:general} deduces Theorem~\ref{thm:general} from
Proposition~\ref{prop:multiwall}. The remaining sections prove
Proposition~\ref{prop:multiwall}: Sections
\ref{sec:lemmas} and~\ref{sec:arith} collect the analytic and
arithmetic inputs{.} Section~\ref{sec:proofmulti} gives the proof.

\subsection{Proof of Theorem~\ref{thm:general}}\label{sec:general}
\begin{proof}[Proof of Theorem~\ref{thm:general}, assuming Proposition~\ref{prop:multiwall}]
Put $H:=\min_j\#I_j$ so that $H \leq \# I_j \leq R H$ for all $j$.  Adjoin to $L$ the bounds of $I$ and set $P=\cP(L)$.
Let $a / q$ with $(a,q) = 1$ be the convergent of $\alpha$ with
the smallest denominator $q \geq 12 R^2 H$. Since the continued fraction expansion
of $\alpha$ has $a_i \leq R$ for all $i \geq 1${, we have} 
$$
q \leq 100 R^3 H.
$$
Notice that $P \subset x + [0, R H]^k$ for some $x \in \mathbb{Z}^k$. 
Pick a smooth function $0 \leq U \leq 1$ such that $U(x) = 1$ for $0 \leq x \leq 1$
and $U$ is compactly supported in $[-2,2]$. 
Therefore the sum we want to bound can be vacuously rewritten as
$$
\sum_{\substack{n \in P}} \prod_{i = 1}^{k} U \Big ( \frac{n_i - x_i}{R H} \Big ) e \Big ( \alpha \sum_{j = 1}^{k} c_j n_j^2 + \beta \cdot n \Big ). 
$$
Write $n = x + m$ so that the constraints defining $P$ take the form $w_i \cdot m - B_i \geq 0$
and the smooth function becomes $\prod_{i = 1}^{k} U \Big ( \frac{m_i}{R H} \Big )$.
Since $|w_i \cdot m| \leq R^2 H$, any such condition with $|B_i| > R^2 H$ is
either vacuous or unsatisfiable for $m \in [0, R H]^k$. 
If a single constraint is unsatisfiable we are done
since the sum is then equal to zero. Vacuous constraints can be removed at no cost.
Therefore we can assume without loss of generality that $|B_i| \leq R^2 H$. 
We notice that since 
$$
\mathbf{1}_{t \geq 0} = \frac{1}{2} \cdot ( 1 + \text{sgn}(2 t + 1) ){,}
$$
the indicator function of $P$ can be expressed as a product of 
$$
\frac{1}{2^{r}} \prod_{i = 1}^{r} \Big ( 1 + \text{sgn} (2 (w_i \cdot m - B_i) + 1) \Big ) \quad \text{with} \quad \ r := \# L + 2k.
$$
Therefore we can re-express our initial sum as
$$
\frac{1}{2^{r}} \sum_{m \in \mathbb{Z}^k} \prod_{i = 1}^{k} U \Big ( \frac{m_i}{R H} \Big ) \prod_{i = 1}^{r} \Big (1 + \text{sgn} ( 2 (w_i \cdot m - B_i) + 1) \Big ) e \Big ( \alpha \sum_{j = 1}^{k} c_j (x_j + m_j)^2 + \beta \cdot (x + m) \Big ).
$$
Expanding the product and the squares $(x_j + m_j)^2$ and then appealing to the Proposition yields the result. 
\end{proof}

\subsection{Preparation for Proposition~\ref{prop:multiwall}: general
lemmas}\label{sec:lemmas}

The next three lemmas seem to be known, but we include them here for completeness.

\begin{lemma}\label{lem:sgn}
Let $L\ge4$ be an even integer. Then for every integer $h$ with $|h|<L/2$,
\begin{equation}\label{eq:saw}
\operatorname{sgn}(h)
=\sum_{\substack{0<|\ell|<L/2\\ \ell\ \mathrm{odd}}}
\widehat c_L(\ell)\,e\Bigl(\frac{\ell h}L\Bigr)
\quad \text{with} \quad
\widehat c_L(\ell)=-\frac{2i}L\cot\Bigl(\frac{\pi\ell}L\Bigr).
\end{equation}
In particular, $|\widehat c_L(\ell)|\ \le\ \frac2{\pi|\ell|}$ for $0<|\ell|<L/2$.
\end{lemma}

\begin{proof}
Apply the finite field Fourier transform to the function $s_L: \mathbb Z/L\mathbb Z \to \mathbb{R}$ defined by $s_L(h)=\operatorname{sgn}(h)$ for $|h|<L/2$ and $s_L(L/2)=0$. Fourier inversion and routine computations imply the desired identity. The bound on the coefficients follows trivially.
\end{proof}

We will also use the following standard application of the Poisson summation formula. 

\begin{lemma}\label{lem:poisson}
Let $U:\mathbb R\to\mathbb C$ be a smooth function supported in $[-2,2]$. 
Given $\lambda,\xi\in\RR$ set
\[
W_U(\lambda;\xi):=\int_{\mathbb R}U(t)\,e\bigl(\lambda t^2+\xi t\bigr)\,{dt}
\]
and, for an integer $q \geq 1$ and $a,b \in \mathbb{Z}$, set
\[
G(a,b;q):=\sum_{x\bmod q}e\Bigl(\frac{ax^2+bx}q\Bigr).
\]
Then for any $N, q \geq 1$, $\delta, \theta \in \RR$ and $a, b \in \mathbb{Z}$,  
\begin{equation}\label{eq:poisson}
\sum_{n\in\mathbb Z}e\Bigl(\Bigl(\frac aq+\delta\Bigr)n^2
+\Bigl(\frac bq+\theta\Bigr)n\Bigr)\,U\Bigl(\frac nN\Bigr)
=\frac Nq\sum_{\ell\in\mathbb Z}G(a,b+\ell;q)\,
W_U\Bigl(\delta N^2;\,N\Bigl(\theta-\frac{\ell}q\Bigr)\Bigr).
\end{equation}
Moreover, for $\Lambda\ge1$, $|\lambda|\le\Lambda$ and $\xi\in\RR$,
\[
|W_U(\lambda;\xi)|\ll_{\Lambda} (1 + \|U'\|_{1} + \|U''\|_{1}) \cdot \frac{1}{1+|\xi|^2}.
\]
\end{lemma}

\begin{proof}
Splitting $n$ into residue classes modulo $q$ and applying Poisson
summation to each class (see e.g.\ \cite[\S4.3]{IK}) gives
\eqref{eq:poisson}. The decay bound follows by integrating by parts
twice in $t$.
\end{proof}

The next lemma records the quadratic dependence of a Gauss sum on its
linear coefficient.
 
\begin{lemma}\label{lem:gaussphase}
Let $c, q \geq 1$ and $(a,q) = 1$. Denote by $\overline{a}$ the inverse
of $a$ modulo $q$. Then, for all $b \in \mathbb{Z}$, 
\[
G(ca,b;q)=\sqrt q \cdot \phi_{a,q}(b;c) \cdot e\Big(-\frac{\overline a\,b^2}{4cq}\Big),
\]
where $x \mapsto \phi_{a,q}(x; c)$ is periodic modulo $2c$ and
$|\phi_{a,q}(x; c)|\leq \sqrt{2c}$ for all $x \in \mathbb{Z}$.
\end{lemma}

\begin{proof}
Substituting $x\mapsto x-\overline a$ and using $a\overline a\equiv1\mymod q$,
we have
\[
ca(x-\overline a)^2+(y+2c)(x-\overline a)\equiv cax^2+yx-(y+c)\overline a \quad
\mymod q,
\]
and therefore
\begin{align*}
G(ca, y + 2c; q) & = \sum_{x \bmod{q}} e \Big ( \frac{ca x^2 + (y + 2 c) x}{q} \Big ) \\
& = \sum_{x \bmod{q}} e \Big ( \frac{ca (x - \overline{a})^2 + (y + 2 c) (x - \overline{a})}{q} \Big ) \\
& = e \Big ( - \frac{(y + c) \overline{a}}{q} \Big ) \sum_{x \bmod{q}} e \Big ( \frac{ca x^2 + y x}{q} \Big )
= e \Big ( - \frac{(y + c) \overline{a}}{q} \Big ) G(ca, y; q).
\end{align*}
Let
\[
\phi_{a,q}(x;c) := \frac{1}{\sqrt{q}} \cdot G(ca, x; q) e \Big ( \frac{\overline{a} x^2}{4 c q} \Big ).
\]
Since $(x+2c)^2-x^2=4c(x+c)$, we find
\begin{align*}
\phi_{a,q}(x + 2c;c) & = \frac{1}{\sqrt{q}} \cdot G(ca, x + 2c; q) e \Big ( \frac{\overline{a} (x + 2c)^2}{4 c q} \Big ) \\
& = \frac{1}{\sqrt{q}} \cdot G(ca, x; q) e \Big ( -\frac{(x + c) \overline{a}}{q} \Big )
e \Big ( \frac{\overline{a} x^2}{4cq}+\frac{(x+c)\overline a}{q} \Big )
= \phi_{a,q}(x;c).
\end{align*}
For the bound, write $|G(ca,b;q)|^2=G(ca,b;q)\,\overline{G(ca,b;q)}$:
\[
|G(ca,b;q)|^2=\sum_{x\bmod q}\ \sum_{z\bmod q}
e\Big(\frac{ca(x^2-z^2)+b(x-z)}{q}\Big).
\]
For fixed $z$, the map $x\mapsto h:=x-z$ is a bijection of $\Z/q\Z$, and
$x^2-z^2=h^2+2hz$. Hence
\[
|G(ca,b;q)|^2=\sum_{h\bmod q}e\Big(\frac{cah^2+bh}{q}\Big)
\sum_{z\bmod q}e\Big(\frac{2cah\,z}{q}\Big).
\]
By orthogonality, the inner sum equals $q$ if $q\mid 2cah$ and $0$
otherwise. Bounding the outer terms by $1$ in absolute value gives
$|G(ca,b;q)|^2\le q\cdot\#\{h\bmod q:\ q\mid 2cah\}$.
Since $(a,q)=1$, the condition $q\mid 2cah$ is equivalent to
$q\mid 2ch$, that is, to $\frac{q}{g}\mid h$ with $g:=\gcd(2c,q)$. There
are exactly $g$ such residues $h$ modulo $q$, so
\[
|G(ca,b;q)|^2\le q\gcd(2c,q)\le 2cq,
\]
and therefore $|\phi_{a,q}(b;c)|=q^{-1/2}|G(ca,b;q)|\le\sqrt{2c}$.
\end{proof}

\subsection{Preparation for Proposition~\ref{prop:multiwall}: an
arithmetic lemma}\label{sec:arith}

Throughout this section $\alpha = [a_0; a_1, a_2, \ldots]$ is an irrational
number with $a_i \leq A$ for all $i \geq 1$, and
$a/q$ is a convergent of $\alpha$, and for $x\in\RR$ we write $\|x\|$ for
the distance from $x$ to the nearest integer.
The following bound follows from a single Weyl differencing step,
combined with the fact that, since the partial quotients of $\alpha$ are
bounded, the fraction $\overline a/q$ is badly approximable at all scales
below $q$. It loses a factor $(\log 2R)^{1/2}$ compared to square-root
cancellation. This loss is harmless for our purposes, since in the proof
of Proposition~\ref{prop:multiwall} the lemma is only applied to sums over
dyadic ranges $[2^J,2^{J+1})$ weighted by $2^{-J}$.

\begin{lemma}\label{lem:weylgauss}
Let $A , M \geq 1$.
Let $\alpha = [a_0; a_1, a_2, \ldots]$ be an irrational number with $a_i \leq A$ for all $i \geq 1$.
Let $a / q$ be a convergent of $\alpha$ and denote by $\overline{a} \in [0,q)$
the modular inverse of $a$ modulo $q$. Then, for all $1 \leq R \leq 2q$,
$$
\sup_{\beta \in \mathbb{R}} \sup_{\substack{1 \leq |m|,|n| \leq M \\ m, n, N \in \mathbb{Z} \\ v, r \geq 1}} \Big | \sum_{\substack{N \leq h < N + R \\ h \equiv v \mymod{r}}} e \Big ( \frac{m}{n} \cdot \frac{\overline{a}}{4 q} \cdot h^2 + \beta h \Big ) \Big | \ll_{M,A} \sqrt{R\log(2R)}.
$$
\end{lemma}
 \begin{proof}
We can assume that $r = 1$, since if the result is known for $r = 1$ it
follows for general $r \geq 1$ using the identity
$$
\mathbf{1}_{h \equiv v \mymod{r}} = \frac{1}{r} \sum_{\ell = 0}^{r - 1} e \Big ( \frac{\ell h}{r} \Big ) e \Big ( - \frac{\ell v}{r} \Big ){,}
$$
which perturbs only the linear term of the exponential sum. Replacing $R$
by $\lceil R\rceil\le2R$ does not change the sum, so we may also assume
that $R$ is an integer with $R\le4q$. If $q<16M$, then $R<64M$ and the
claim follows from the trivial bound, so we assume that $q\ge16M$.
Splitting the range of summation into at most $32M+1$ intervals, each
consisting of at most $q/(4M)$ consecutive integers, and noting that
$\sqrt{R\log(2R)}$ is increasing in $R$, we may further assume that
$R\le q/(4M)$. Finally, the substitution $h\mapsto h+N$ only changes the
linear coefficient $\beta$ and introduces a unimodular constant factor, so
we may assume that $N=0$.

Write $p_i / q_i$ for the sequence of convergents of $\alpha$
so that by assumption $a / q = p_s / q_s$ for some $s \geq 0$. Since
$q\geq16M>1=q_0$, we have $s\geq1$.
Then, as is well known (see \cite[Chapter~X]{HW})
\begin{enumerate}
\item $p_i q_{i- 1} - p_{i - 1} q_i = (-1)^{i - 1}$
for all $i \geq 0$, with the convention $p_{-1} = 1$ and $q_{-1} = 0${.}
\item $q_{i - 1} / q_i = [0; a_i, a_{i - 1}, \ldots, a_1]$ for all $i\geq1$.
\end{enumerate}
In particular, $a q_{s - 1} - p_{s - 1} q = (-1)^{s - 1}$, and thus,
$$
\overline{a} \equiv \pm q_{s - 1} \quad \mymod{q},
$$
and
\begin{equation} \label{eq:seriescont}
\frac{q_{s - 1}}{q} = [0;a_{s}, a_{s - 1}, \ldots, a_{1}].
\end{equation}
Let $1 = w_0 \leq w_1 < w_2 < {\cdots} < w_s = q$ be the sequence
of denominators of the convergents of \eqref{eq:seriescont}.
Since $a_i \leq A$ we have $w_{j + 1} \leq (A + 1) w_{j}$ for $0\leq j<s$.
Moreover, for any integer $1 \leq t < q$ we can find a
$j$ with $w_{j} \leq t < w_{j + 1}$, and then, by the best approximation
property of convergents,
$$
\Big \| t \cdot \frac{\overline{a}}{q} \Big \|
= \Big \| t \cdot \frac{q_{s-1}}{q} \Big \|
\geq \Big \| w_j \cdot \frac{q_{s-1}}{q} \Big \|
\geq \frac{1}{w_j + w_{j + 1}} \geq \frac{1}{(A + 2) t}{.}
$$
Put
$$
\alpha':=\frac{m}{n}\cdot\frac{\overline a}{4q}
\qquad\text{and}\qquad
c:=\frac{1}{12AM^2}.
$$
If $t$ is an integer with $1\leq|t|\leq q/(2M)$, then $1\leq|tm|\leq q/2<q$.
Since moreover $\|x\|\leq4|n|\,\|x/(4n)\|$ for every $x\in\RR$, we obtain
\begin{equation} \label{eq:dioph}
\| t\alpha' \| \geq \frac{ \| t m \overline{a} / q \|}{4 |n|} \geq \frac{1}{4 (A + 2) |t m n|} \geq \frac{c}{|t|}
\qquad\Big(1\leq|t|\leq\frac{q}{2M}\Big).
\end{equation}
It now remains to {bound}
\begin{equation} \label{eq:quadraticsum}
S:=\sum_{0 \leq h < R} e \big ( \alpha' h^2 + \beta h \big ).
\end{equation}
Expanding $|S|^2$ and writing $h=h'+d$, so that $h^2-h'^2=2dh'+d^2$, we get
$$
|S|^2=\sum_{|d|<R}e\big(\alpha'd^2+\beta d\big)
\sum_{\substack{0\leq h'<R\\ 0\leq h'+d<R}}e\big(2\alpha'd\,h'\big).
$$
The inner sum runs over an interval of at most $R$ integers, so it is
bounded in absolute value by $\min\{R,(2\|2d\alpha'\|)^{-1}\}$. The terms
$d$ and $-d$ obey the same bound, hence
\begin{equation}\label{eq:weyldiff}
|S|^2\leq R+\sum_{1\leq d<R}\|2d\alpha'\|^{-1},
\end{equation}
where all terms are finite by \eqref{eq:dioph}, as $2d<2R\leq q/(2M)$.
Put $x_d:=2d\alpha'\in\RR/\Z$ for $1\leq d<R$. If $1\leq d<d'<R$, then
$1\leq2(d'-d)<2R\leq q/(2M)$, so \eqref{eq:dioph} gives
$$
\|x_{d'}-x_d\|=\|2(d'-d)\alpha'\|\geq\frac{c}{2(d'-d)}>\frac{c}{2R}
\qquad\text{and similarly}\qquad
\|x_d\|\geq\frac{c}{2d}>\frac{c}{2R}.
$$
Thus the points $x_d$ are $c/(2R)$-separated in $\RR/\Z$, and none of them
lies within distance $c/(2R)$ of $0$. Consequently, for every integer
$j\geq1$, the points $x_d$ with $\|x_d\|<jc/(2R)$ lie in the union of two
half-open arcs of length $(j-1)c/(2R)$, each of which contains at most $j-1$
of them. Ordering the values $\|x_d\|$ increasingly as
$\delta_1\leq\delta_2\leq\cdots\leq\delta_{R-1}$, it follows that
$\delta_t\geq\lceil t/2\rceil\,c/(2R)$, and therefore
$$
\sum_{1\leq d<R}\|2d\alpha'\|^{-1}
=\sum_{t=1}^{R-1}\delta_t^{-1}
\leq\frac{2R}{c}\sum_{t=1}^{R}\frac{1}{\lceil t/2\rceil}
\leq\frac{4R}{c}\,(1+\log R).
$$
Inserting this into \eqref{eq:weyldiff} gives
$|S|^2\leq R+48AM^2R(1+\log R)\ll_{A,M}R\log(2R)$, and the lemma follows.
\end{proof}

\subsection{Proof of Proposition~\ref{prop:multiwall}}\label{sec:proofmulti}

\begin{proof}
We aim to bound
\[
T:=\sum_{h\in\Z^k}\ \prod_{i=1}^r\operatorname{sgn}\bigl(2(w_i\cdot h-B_i)+1\bigr)
\prod_{j=1}^kU\Bigl(\frac{h_j}P\Bigr)\,
e\Bigl(\alpha\sum_{j=1}^k c_j h_j^2+\beta\cdot h\Bigr).
\]
By hypothesis, $|2(w_i\cdot h-B_i)+1|\le q/2+1<q$ on the support of
$\prod_jU(h_j/P)$. Hence Lemma~\ref{lem:sgn} with $L=2q$ gives the
Fourier expansion
\[
\operatorname{sgn}\bigl(2(w_i\cdot h-B_i)+1\bigr)
=\sum_{\substack{|v| < q}}\widehat c_{2q}(v)\,
e\Bigl(\frac{(1-2B_i)v}{2q}\Bigr)\,
e\Bigl(\frac{v\,w_i\cdot h}q\Bigr),
\]
with $\widehat{c}_{2q}(v) = 0$ for even $v$.
Plugging this in for each $1\le i\le r$ and expanding the product, the
sum over $h$ factors over the coordinates $h_1,{\ldots},h_k$, and we find that
\begin{equation} \label{eq:mainlong}
T=\sum_{v \in \mathbb{Z}^r \cap (-q,q)^{r}}\Big (\prod_{i=1}^r\widehat c_{2q}(v_i) e \Big ( \frac{(1 - 2B_i) v_i}{2q} \Big )\Big )
\prod_{j=1}^k S_j \Big ( \beta_j + \frac{1}{q} \sum_{i = 1}^{r} v_i w_{ij}  \Big ),
\end{equation}
where $w_{i} = (w_{i1}, \ldots, w_{ik})$ and
$$
S_j(\gamma):=\sum_{h\in\Z}U\Big (\frac hP\Big )
e\Bigl(\alpha c_jh^2+\gamma h\Bigr).
$$
Recall that by assumption $a/q$ denotes a convergent of $\alpha$.
Let 
$$
\delta := \alpha - \frac{a}{q} \text{ so that } |\delta| \leq \frac{1}{q^2}. 
$$
By Lemma \ref{lem:poisson}, 
\begin{equation} \label{eq:poissonapplied}
S_j(\gamma) = \frac{P}{q} \sum_{\ell \in \mathbb{Z}} G(c_j a, \ell; q) W_U \Big ( c_j \delta P^2 ; P \Big ( \gamma - \frac{\ell}{q} \Big ) \Big )
\end{equation}
with 
$$
W_U(\lambda;\xi) := \int_{\mathbb{R}} U(t)\, e\big(\lambda t^2 + \xi t\big)\, dt.
$$
In particular, 
\begin{align*}
S_j \Big (\beta_j + \frac{1}{q} \sum_{i = 1}^{r} v_i w_{i j} \Big ) & = \frac{P}{q} \sum_{\ell \in \mathbb{Z}} G(c_j a, \ell ; q) W_U \Big ( c_j \delta P^2; P \Big ( \beta_j + \frac{1}{q} \sum_{i = 1}^{r} v_i w_{i j} - \frac{\ell}{q} \Big ) \Big ) \\ & = \frac{P}{q} \sum_{\ell \in \mathbb{Z}} G(c_j a, \ell + g_j(v); q) W_U \Big ( c_j \delta P^2; P \Big ( \beta_j - \frac{\ell}{q} \Big ) \Big ),
\end{align*}
where $g_j(v) := \sum_{i = 1}^{r} v_i w_{i j}$.
Plugging the above equation into \eqref{eq:mainlong} and interchanging sums{, we obtain} 
\begin{equation} \label{eq:exactformula}
T=\Big(\frac {P}{\sqrt{q}}\Big)^k\sum_{\ell \in\Z^k}
\Big(\prod_{j=1}^kW_U\Big(c_j\delta P^2;\,P\Big(\beta_j-\frac{\ell_j}{q}\Big)\Big)\Big) \cdot \Sigma_{\ell},
\end{equation}
where
$$
\Sigma_{\ell}:=\sum_{v\in\Z^r\cap(-q,q)^r}
\Big(\prod_{i=1}^r\widehat c_{2q}(v_i) e \Big ( \frac{(1 - 2B_i) v_i}{2q} \Big ) \Big)
\prod_{j=1}^k \frac{G (c_ja,\ell_j + g_j(v);q)}{\sqrt{q}}.
$$
The trivial bound for $\Sigma_{\ell}$ is $\ll (\log q)^{r}$. We will show the
stronger {bound}
$$
\sup_{\ell \in \mathbb{Z}^k} |\Sigma_{\ell}| \ll 1. 
$$
Indeed with this bound at hand \eqref{eq:exactformula} is
\begin{align*}
& \ll \prod_{j = 1}^{k} \Big ( \frac{P}{\sqrt{q}} \sum_{\ell \in \mathbb{Z}} \Big | W_U \Big ( c_j \delta P^2 ; P \Big ( \beta_j - \frac{\ell}{q} \Big ) \Big ) \Big | \Big ) \\
& \ll \prod_{j = 1}^{k} \Big ( \frac{P}{\sqrt{q}} \sum_{\ell \in \mathbb{Z}} \Big ( 1 + P \Big | \beta_j - \frac{\ell}{q} \Big | \Big )^{-2} \Big ) \ll \Big ( \frac{P}{\sqrt{q}} \cdot \frac{q}{P} \Big )^{k} = q^{k / 2}.
\end{align*}
It remains to prove that $\Sigma_\ell\ll1$ uniformly in $\ell\in\Z^k$.
Let $\overline a\in[0,q)$ be the inverse of $a$ modulo $q$, and write
$\phi_{a,q}(\cdot\,;c_j)$ for the function given by
Lemma~\ref{lem:gaussphase}. Then
\begin{equation} \label{eq:big}
\Sigma_\ell=\sum_{v\in\Z^r\cap(-q,q)^r}
\Big(\prod_{i=1}^r\widehat c_{2q}(v_i)\,e\Big(\frac{(1-2B_i)v_i}{2q}\Big)\Big)
\prod_{j=1}^k\phi_{a,q}(\ell_j+g_j(v); c_j )
e\Big(-\frac{\overline a}{4q}\cdot\frac{(\ell_j+g_j(v))^2}{c_j}\Big).
\end{equation}
Let $Q = 2 {c_1 \cdots c_k}$. We fix each component of $v$ in a residue class
modulo $Q$. On each such congruence class the product of the $\phi_{a,q}(\ell_j + g_j(v); c_j)$ is
fixed. 
Therefore \eqref{eq:big} is
\begin{equation} \label{eq:crude}
\ll \sup_{\rho \in \{0,{\ldots},Q-1\}^r} \Big |
\sum_{\substack{v \in \mathbb{Z}^r \cap (-q, q)^r \\ v_i \equiv \rho_i \ (\mathrm{mod}\ Q)}}
\Big(\prod_{i = 1}^{r} \widehat{c}_{2q}(v_i)\, e \Big ( \frac{(1 - 2B_i) v_i}{2q} \Big )\Big)
\prod_{j=1}^k e \Big ( - \frac{\overline{a}}{4 q} \cdot \frac{(\ell_j + g_j(v))^2}{c_j} \Big ) \Big |.
\end{equation}
Expanding the squares,
$$
\sum_{j=1}^k\frac{(\ell_j+g_j(v))^2}{c_j}
=\sum_{i=1}^r\Big(\sum_{j=1}^k\frac{w_{ij}^2}{c_j}\Big)v_i^2
+\sum_{i\ne i'}\Big(\sum_{j=1}^k\frac{w_{ij}w_{i'j}}{c_j}\Big)v_iv_{i'}
+2\sum_{i=1}^r\Big(\sum_{j=1}^k\frac{\ell_jw_{ij}}{c_j}\Big)v_i
+\sum_{j=1}^k\frac{\ell_j^2}{c_j}.
$$
In particular, once all coordinates except $v_{i}$ are fixed, the
phase is
$$
-\frac{m_{i}}{{c_1 \cdots c_k}}\cdot\frac{\overline a}{4q}\,v_{i}^2+\gamma\,v_{i}+\theta
{\quad \text{with} \quad} m_{i}:=\sum_{j=1}^kw_{i j}^2\,\frac {{c_1 \cdots c_k}}{c_j},
$$
for some $\theta, \gamma\in\RR$ depending on the 
coordinates $v_j$ with $j \neq i$.
We now split the sum in \eqref{eq:crude} into 
$$
\sum_{\varepsilon_1, \ldots, \varepsilon_r \in \{-1,1\}} \sum_{2^{J_1}, \ldots, 2^{J_r} \leq q} S_{\varepsilon_1 J_1, \ldots, \varepsilon_r J_{r}}{,}
$$
where
$$
S_{\varepsilon_1 J_1, \ldots, \varepsilon_r J_r} := \sum_{\substack{v \in \mathbb{Z}^r \cap (-q, q)^r \\ v_i \equiv \rho_i \ (\mathrm{mod}\ Q) \\ 2^{J_i} \leq \varepsilon_i v_i < 2^{J_{i} + 1}}}
\Big(\prod_{i = 1}^{r} \widehat{c}_{2q}(v_i)\, e \Big ( \frac{(1 - 2B_i) v_i}{2q} \Big )\Big)
\prod_{j=1}^k e \Big ( - \frac{\overline{a}}{4 q} \cdot \frac{(\ell_j + g_j(v))^2}{c_j} \Big ).
$$
We now make the following observation: For each such sum, pick an index $1 \leq i \leq r$ and freeze all
variables $v_{j}$ with $j\ne i$. As a function of $v_i$, the phase is then
$$-\frac{m_i}{c_1\cdots c_k}\cdot\frac{\overline a}{4q}v_i^2+\gamma v_i+\kappa,$$
and $v_i$ runs over the progression $v_i\equiv\rho_i\mymod Q$ in the
range $2^{J_i}\le\varepsilon_iv_i<2^{J_i+1}$, which has length
$2^{J_i}<q$. Absorbing all terms linear in $v_i$ into $\beta$, we obtain
\begin{align*}
|S_{\varepsilon_1J_1, \ldots,\varepsilon_rJ_r}| 
\le \sum_{\substack{(v_j)_{j\ne i}\\ v_j \equiv \rho_j \ (\mathrm{mod}\ Q) \\ 2^{J_j} \leq \varepsilon_j v_j < 2^{J_{j} + 1}}}
\Big(\prod_{j \neq i} |\widehat{c}_{2q}(v_j)| \Big)
 \sup_{\beta \in \mathbb{R}} \Big | \sum_{\substack{v_i \equiv \rho_i \mymod{Q} \\ 2^{J_i} \leq \varepsilon_i v_i < 2^{J_i + 1}}}
\widehat{c}_{2q}(v_i)\, e \Big ( - \frac{m_i}{c_1 \cdots c_k} \cdot \frac{\overline{a}}{4 q}\, v_i^2 + \beta v_i \Big ) \Big |.
\end{align*}
Since $i \widehat c_{2q}$ is real and monotone on the range of $v_i$
and bounded by $2^{1-J_i}$ (Lemma~\ref{lem:sgn}), Abel summation followed by Lemma \ref{lem:weylgauss}
(applied with $m=-m_i$, $n=c_1\cdots c_k$, $M=V^2C^k$, modulus $Q$, and
intervals of length at most $2^{J_i}\leq q$) gives
\begin{align*}
\sup_{\beta\in\RR} & \Big|\sum_{\substack{v_i\equiv\rho_i\ (\mathrm{mod}\ Q)\\ 2^{J_i}\le\varepsilon_iv_i<2^{J_i+1}}}
\widehat c_{2q}(v_i)\,e\Big(-\frac{m_i}{c_1\cdots c_k}\cdot\frac{\overline a}{4q}\,v_i^2+\beta v_i\Big)\Big|
\\ & \ll 2^{-J_i}\sup_{\beta\in\RR}\ \sup_{I}\Big|\sum_{\substack{v_i\in I\\ v_i\equiv\rho_i\ (\mathrm{mod}\ Q)}}
e\Big(-\frac{m_i}{c_1\cdots c_k}\cdot\frac{\overline a}{4q}\,v_i^2+\beta v_i\Big)\Big|
\\ & \ll 2^{-J_i}\bigl(2^{J_i}\log 2^{J_i+1}\bigr)^{1/2}
\ll (1+J_i)^{1/2}\,2^{- J_i / 2}{,}
\end{align*}
where $I$ runs over the subintervals of
$\{v:\ 2^{J_i}\le\varepsilon_iv<2^{J_i+1}\}$.
Since for all $j \neq i$ {we have}
 $$\sum_{2^{J_j}\le\varepsilon_jv<2^{J_j+1}}|\widehat c_{2q}(v)|\ll1{,}$$
it follows that for every $1 \leq i \leq r$ we have $|S_{\varepsilon_1J_1,\ldots,\varepsilon_rJ_r}|\ll(1+J_i)^{1/2}2^{-J_i/2}$.
In particular, choosing $i$ with $J_i=\max_{1\leq i'\leq r}J_{i'}$,
$$
|S_{\varepsilon_1 J_1, \ldots, \varepsilon_r J_r}| \ll (1+u)^{1/2}\,2^{-u/2}
\qquad\text{with}\qquad u:=\max_{1 \leq i \leq r} J_i.
$$
There are at most $r(u+1)^{r-1}$ tuples $(J_1,\ldots,J_r)$ of non-negative
integers with $\max_iJ_i=u$. Summing over all the $J_1, \ldots, J_r$ and
over the $2^r$ choices of signs, we obtain
$$
\sup_{\ell \in \mathbb{Z}^k} |\Sigma_{\ell}| \ll \sum_{u \geq 0} \frac{(u+1)^{r - 1/2}}{2^{u/2}} \ll 1.
$$
This finishes the proof.
\end{proof}

\section{The decomposition theorem and proof of Theorem \ref{thm:box}}\label{sec:decomp}

This section proves a decomposition theorem which, combined with Theorem~\ref{thm:general}
from the previous section, yields the proof of Theorem~\ref{thm:box} at the end of this section.
Recall that a box is a product of finite intervals of integers, that $\cP(W)$ is the set of
$n\in\z^k$ with $w\cdot n+b\ge0$ for all $(w,b)\in W$, and that
\[
  \vol(J)=\prod_j\#J_j\qquad\text{and}\qquad
  \len(w;J)=\max\bigl(1,\max_{j\in\supp w}\#J_j\bigr).
\]
For a finite $W\subset\z^k\times\z$ we write
\[
 \rho(W)=\max\Bigl\{\#W,\ \max_{(w,b)\in W}|w|_1\Bigr\},
\]
and associate to $W$ the undirected graph $\Gamma_k(W)$ with vertices $\{1, \ldots, k\}$, in which $i\neq j$ are joined
by an edge whenever they occur together in some inequality of $W$, i.e.\ $w_iw_j\neq0$ for some $(w,b)\in W$.

\begin{definition}[Componentwise balanced boxes]\label{def:balanced-W}
Let $k, R\ge1$ and let $W\subset(\z^k\setminus\{0\})\times\z$ be finite. A box $J=\prod_{j=1}^kJ_j$
is \emph{$R$-balanced with respect to $W$} if
\[
 \#J_i\le R\#J_j
 \quad\text{whenever }i,j\text{ belong to the same component of }
 \Gamma_k(W).
\]
\end{definition}

\begin{theorem}[Decomposition with balanced enclosing boxes]\label{thm:decomp}
Let $k,V,h\ge1$ and $r\ge0$ be integers, and let $c_1,\ldots,c_k$ be positive integers with
$c_j\le h$. Let $W\subset\z^k\times\z$ with $\#W\le r$ be such that $0<|w|_1\le V$ for all $(w,b)\in W$.
Suppose that $P=\cP(W)$ is finite and nonempty, and let $I$ be the smallest box containing $P$.
Then, there exists an integer $R=R(k,r,V,h)$ and a disjoint decomposition
\begin{equation}\label{eq:decomposition}
 P=\bigsqcup_{m=1}^M\bigl(u_m+\Phi_m\cP(W_m)\bigr)
\end{equation}
with the property that $M\le R$, $u_m\in\z^k$, $\Phi_m\in\z^{k\times k}$ is non-singular, each $W_m \subset \z^k \setminus\{0\} \times \z$ is finite, the boxes $J_m$ are $R$-balanced with respect to
$W_m$ and
\begin{equation}\label{eq:boxed-conclusion}
 \cP(W_m)\subseteq J_m \quad \text{and} \quad  \vol(J_m)\le R\vol(I).
\end{equation}
Furthermore each
$\Phi_m^{\mathsf T}\operatorname{diag}(c_j)\Phi_m$ is diagonal, and
\[
\max\Bigl(|\Phi_m|_\infty ,\, |\Phi_m^{\mathsf T}\operatorname{diag}(c_j)\Phi_m|_\infty ,\, \rho(W_m)\Bigr) \leq R
\qquad(1\le m\le M).
\]
If moreover both $(w_0,b_0)$ and $(-w_0,-b_0)$ belong to $W$, then the decomposition
\eqref{eq:decomposition} can be chosen so that, in addition,
\begin{equation}\label{eq:decomposition-section}
 \vol(J_m)\le R\,\frac{\vol(I)}{\len(w_0;I)}\qquad(1\le m\le M).
\end{equation}
\end{theorem}

We now briefly outline the main idea in the proof. 
Let $J\supseteq\cP(W)$ be a box with shortest side $H$, and call a coordinate $j$ of $J$
\emph{short} if $\#J_j\ll H$ and \emph{long} otherwise. If all coordinates of $J$ are short, they
are all comparable in size to $H$ and we are done. Otherwise, there are two possibilities: either no
inequality in $W$ mixes short and long coordinates, or there exists an inequality that mixes short
and long coordinates of $J$. In the first case, after reordering the coordinates we can 
write $\cP(W) = \cP(W_1) \times \cP(W_2)$ with both factors of lower dimension and we recurse. 
In the second case, some inequality $w \cdot x+b\geq 0$ with $(w,b) \in W$ involves both long and short
coordinates. Let $B$ be the set of short coordinates of $J$
and write $w=w_B+w_{B^c}$, with $w_B$ supported on $B$ and $w_{B^c}$ on its complement. Then
we split $\cP(W)$ into two parts and recurse on both: on the first, the inequality $w \cdot x + b \geq 0$ can be replaced 
by one involving only $w_{B^c}\cdot x${, while} on the second, $w_{B^c}\cdot x$ varies by only $O(H)$
(Figure~\ref{fig:split}), and Lemma~\ref{lem:short-coordinate}, after an affine change of variables,
turns $w_{B^c} \cdot x$ into a coordinate of side $H$. In either case we lower 
the triple 
$$(\text{dimension}, \text{number of long coordinates}, \text{number of mixed inequalities})$$
lexicographically, and every change of variables keeps $\sum_jc_jx_j^2$ diagonal.
The recursion resembles Lenstra's algorithm for integer programming in fixed
dimension~\cite{Len}, which slices a flat set along its thin direction{. Here} the
thin direction is instead made a coordinate, by changes of variables that keep
$\sum_jc_jx_j^2$ diagonal.

\subsection{Two tools}\label{sec:tools}

\begin{definition}[Pieces]\label{def:piece}
A \emph{piece} of dimension $d\ge1$ is a quadruple $(W,J,Q,A)$, where $W\subset\Z^d\times\Z$ is finite,
$J=\prod_{j=1}^dJ_j$ is a box, $Q(x)=\sum_{j=1}^d\gamma_jx_j^2$ with integers $\gamma_j\ge1$, and
$A\ge1$ is an integer, such that
\[
 \cP(W)\subseteq J\qquad\text{and}\qquad\max\{d,\rho(W),|(\gamma_j)|_\infty\}\le A.
\]
Let $H=\min_{1\le j\le d}\#J_j$. A coordinate $j$ is said to be \emph{short} if $\#J_j\le AH$ and \emph{long}
otherwise. We define $B:=\{j : \#J_j \leq AH\}$ and $B^c=\{1,\ldots,d\}\setminus B$ as the sets of short and long coordinates. An
inequality $(w,b)\in W$ is said to be \emph{mixed} if $\supp w$ has nonempty intersection with both $B$ and $B^c$.
The \emph{index} of the piece is $(d,n_{\mathrm{long}},n_{\mathrm{mix}})$, where $n_{\mathrm{long}}=\#B^c$ and $n_{\mathrm{mix}}$ is the number of mixed inequalities in $W$. Indices are ordered lexicographically.
\end{definition}

Clearly $n_{\mathrm{mix}}\le\#W\le A$ and, since $B\ne\varnothing$, we have $n_{\mathrm{long}}<d$.

\begin{definition}[Transitions, size and decompositions]\label{def:transition}
Let $(W,J,Q,A)$ and $(W',J',Q',A')$ be pieces of dimension $d$. A \emph{transition} from
$(W,J,Q,A)$ to $(W',J',Q',A')$ is a substitution $x=u+\Phi y$, with $u\in\Z^d$ and
$\Phi\in\Z^{d\times d}$ non-singular, such that
\[
 u+\Phi\,\cP(W')\subseteq\cP(W)\qquad\text{and}\qquad Q'(y)=Q(\Phi y).
\]
We define the size of a transition by $|\Phi|_\infty$. A \emph{decomposition} of $(W,J,Q,A)$ is a finite family of
transitions $x=u_m+\Phi_my$ from $(W,J,Q,A)$ to pieces $(W_m,J_m,Q_m,A_m)$ such that
\[
\cP(W)=\bigsqcup_m\bigl(u_m+\Phi_m\cP(W_m)\bigr).
\]
\end{definition}

Since $Q(u+\Phi y)=Q'(y)+2(\Phi^{\mathsf T}Gu)\cdot y+Q(u)$, where
$G=\operatorname{diag}(\gamma_1,{\ldots},\gamma_d)$, after the substitution $x=u+\Phi y$ the form $Q$ is
diagonal in $y$ up to terms of degree at most $1$.

\begin{lemma}[Composing and multiplying transitions]\label{lem:compose}
The following holds true:
\begin{enumerate}
\item If $x=u+\Phi y$ is a transition from $(W,J,Q,A)$ to $(W',J',Q',A')$ and $y=v+\Psi z$ is one
from $(W',J',Q',A')$ to $(W'',J'',Q'',A'')$, then $x=(u+\Phi v)+\Phi\Psi z$ is a transition from
$(W,J,Q,A)$ to $(W'',J'',Q'',A'')$ of size at most $d\,|\Phi|_\infty|\Psi|_\infty$. Consequently, a
decomposition of a piece, followed by decompositions of each of the resulting pieces, gives a
valid decomposition of the original piece.
\item If $(W_i,J_i,Q_i,A_i)$ is a piece of dimension $d_i$ for $i=1,2$, then
$(W_1\oplus W_2,\,J_1\times J_2,\,Q_1\oplus Q_2,\,A_1+A_2)$ is a piece of dimension $d_1+d_2$, where
$W_1\oplus W_2$ is the union of $W_1$ and $W_2$ with the vectors $w$ padded with zeros and
$(Q_1\oplus Q_2)(y_1,y_2)=Q_1(y_1)+Q_2(y_2)$. Transitions from $P_i := (W_i,J_i,Q_i,A_i)$ to
$P'_i := (W'_i,J'_i,Q'_i,A'_i)$, $i=1,2$, combine coordinatewise into a transition from the product of 
$P_1$ and $P_2$ to the product of $P_1'$ and $P_2'$. Consequently, decompositions of two lower-dimensional 
pieces combine into a decomposition of their product.
\end{enumerate}
\end{lemma}

\begin{proof}
Both parts follow directly from the definitions.
\end{proof}

\begin{lemma}[Splitting along a mixed inequality]\label{lem:split}
Let $(W,J,Q,A)$ be a piece of index $(d,n_{\mathrm{long}},n_{\mathrm{mix}})$ with $B$ the set of 
short coordinates, and let $(w,b)\in W$ be mixed. Write $w=w_B+w_{B^c}$, 
where $w_B$ is obtained from $w$ by setting its entries outside $B$ to $0$,
and $w_{B^c}=w-w_B$, and let $\lambda_{w,b}=\min_{x\in J}(w_B\cdot x+b)$. Let $W^+$ be
obtained from $W$ by replacing $(w,b)$ with $(w_{B^c},\lambda_{w,b})$ and adjoining the $2d$ inequalities
$x_j\ge\min J_j$ and $x_j\le\max J_j$, and let $W^-=W\cup\{(-w_{B^c},-\lambda_{w,b}-1)\}$. Then
$(W^+,J,Q,A+2d)$ and $(W^-,J,Q,A+1)$ are pieces {with the following properties:}
\begin{enumerate}
\item The identity map gives transitions from $(W,J,Q,A)$ to $(W^+,J,Q,A+2d)$ and from $(W,J,Q,A)$ to $(W^-,J,Q,A+1)$, forming a
decomposition of $(W,J,Q,A)$;
\item The piece $(W^+,J,Q,A+2d)$ has index strictly smaller than $(d,n_{\mathrm{long}},n_{\mathrm{mix}})$;
\item We have $|w_{B^c}\cdot(x-x')|<A^2H$ for $x,x'\in\cP(W^-)$.
\end{enumerate}
\end{lemma}

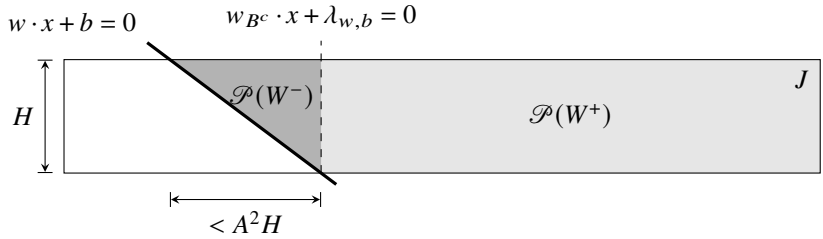
\begin{figure}[ht]
\centering
\begin{tikzpicture}[x=1cm,y=1cm,font=\footnotesize]
 \fill[black!10] (3.4,0) rectangle (10,1.5);
 \fill[black!30] (3.4,0) -- (3.4,1.5) -- (1.4,1.5) -- cycle;
 \draw (0,0) rectangle (10,1.5);
 \draw[very thick] (3.6,-0.15) -- (1.1,1.725);
 \draw[dashed] (3.4,0) -- (3.4,1.75);
 \node[anchor=south east] at (1.1,1.725) {$w\cdot x+b=0$};
 \node[anchor=south] at (3.4,1.75) {$w_{B^c}\cdot x+\lambda_{w,b}=0$};
 \node at (2.75,1.0) {$\cP(W^-)$};
 \node at (6.7,0.75) {$\cP(W^+)$};
 \node[anchor=north east] at (10,1.5) {$J$};
 \draw[|<->|,>=stealth] (-0.25,0) -- (-0.25,1.5) node[midway,left] {$H$};
 \draw[|<->|,>=stealth] (1.4,-0.35) -- (3.4,-0.35) node[midway,below] {$<A^2H$};
\end{tikzpicture}
\caption{Lemma~\ref{lem:split} for $d=2$, with $x_1$ long and $x_2$ short.}\label{fig:split}
\end{figure}

\begin{proof}
Since $W^{+}$ is obtained from $W$ by adjoining $2d$ conditions equivalent to $x \in J$ and replacing
$(w,b)$ by $(w_{B^c}, \lambda_{w,b})$, with $|w_{B^c}|_1 \leq |w|_1 \leq A$, we have
$\cP(W^+) \subseteq J$ and $\rho(W^{+}) \leq A + 2d$. Clearly $\cP(W^{-}) \subseteq \cP(W) \subseteq J$,
and $W^{-}$ is obtained from $W$ by adjoining only one condition $(-w_{B^c}, - \lambda_{w,b} - 1)$,
so $\# W^{-} \leq \# W + 1$ and $\rho(W^{-}) \leq \rho(W) + 1 \leq A + 1$. As $d$ and $Q$ are
unchanged, $(W^{+}, J, Q, A + 2d)$ and $(W^{-}, J, Q, A + 1)$ are pieces.

(1) As the transitions are identities, we must show that $\cP(W)=\cP(W^+)\sqcup\cP(W^-)$. We have
seen that $\cP(W^{-}) \subseteq \cP(W)$. If $w_{B^c} \cdot x + \lambda_{w,b} \geq 0$
and $x \in J$ then $w \cdot x + b \geq 0$. Therefore, as $\cP(W^+) \subseteq J$, we have
$\cP(W^+) \subseteq \cP(W)$.
Since the two inequalities $w_{B^c} \cdot x + \lambda_{w,b} \geq 0$ and $(-w_{B^c}) \cdot x + (-\lambda_{w,b} - 1) \geq 0$
cannot both hold for the same $x \in \mathbb{Z}^d$ we see that $\cP(W^{-})$ and $\cP(W^{+})$ are disjoint. It remains to 
show that if $x \in \cP(W)$ then either $x \in \cP(W^{-})$ or $x \in \cP(W^{+})$. 
Since $x \in \cP(W)$ and $\cP(W) \subset J$ we have the $2d$ conditions $\min J_i \leq x_i$ and $x_i \leq \max J_i$. 
If $w_{B^c} \cdot x + \lambda_{w,b} \geq 0$ holds then $x \in \cP(W^+)$ while if $w_{B^c} \cdot x + \lambda_{w,b} \geq 0$ fails
then $w_{B^c} \cdot x + \lambda_{w,b} \leq -1$ and $x \in \cP(W^{-})$. 

(2) Let $H := \min_i \# J_i$. We note that all sides with $\#J_i \leq (A + 2d)H$ are considered short. Therefore $n_{\text{long}}$ can decrease in $(W^+, J, Q, A + 2d)$ compared to $(W, J, Q, A)$ and if it does we are done since that decreases $(d, n_{\text{long}}, n_{\text{mix}})$ lexicographically. Otherwise if $n_{\text{long}}$ stays the same the set $B$ of short coordinates in both pieces is the same and in that event $n_{\text{mix}}$ has decreased in $(W^+, J, Q, A + 2d)$ compared to $(W, J, Q, A)$ since we replaced $(w,b)$ by $(w_{B^c}, \lambda_{w,b})$ and none of the $2d$ inequalities $\min J_i \leq x_i$ or $x_i \leq \max J_i$ affects $n_{\text{mix}}$. 

(3) For $x \in \cP(W^{-})$ we have both 
$$
w \cdot x + b \geq 0 \quad \text{ and } \quad w_{B^c} \cdot x + \lambda_{w,b} \leq - 1{.} 
$$
In particular, 
$$
- (w_{B} \cdot x + b) \leq w_{B^c} \cdot x  
$$
for all $x \in \cP(W^{-})$, and also 
$$
w_{B^c} \cdot x \leq - 1 - \lambda_{w,b}.
$$
Therefore, for all $x \in \cP(W^{-})$,
$$
- \max_{x \in J} (w_B \cdot x + b) \leq w_{B^c} \cdot x \leq - \min_{x \in J} (w_B \cdot x + b) - 1{.}
$$
It follows that, for any $x, x' \in \cP(W^{-})$, 
$$
|w_{B^c} \cdot (x - x')| \leq \max_{x \in J} (w_B \cdot x + b) - \min_{x \in J} (w_B \cdot x + b) - 1 < A^2 H 
$$
since $|w_B|_1 \leq A$ and $w_{B}$ is non-zero only on coordinates of $x$ that lie in an interval of length $\leq A H$.
\end{proof}

\begin{lemma}[A short linear combination as a coordinate]\label{lem:short-coordinate}
Let $(W,J,Q,R)$ be a piece of dimension $d$, let $H$ be an integer with $1\le H\le\min_j\#J_j$, and let
$a\in\Z^d\setminus\{0\}$ satisfy $|a|_1\le R$ and $|a\cdot(x-x')|\le RH$ for $x,x'\in\cP(W)$.
Choose $q\in\supp a$ with largest $\#J_q$. Then there exist a set $U\subset\Z^d$ with $\# U\ll_R 1$ and a matrix
$D\in\Z^{d\times d}$ such that the transitions $x=u+Dy$, $u\in U$, give a
decomposition of $(W,J,Q,R)$ into pieces $(W'_u,J'_u,Q',R')$ with $R'\ll_R1$ and
\[
 \#J'_{u,q}=H,\qquad \#J_j\le\#J'_{u,j}\le3\#J_j\ (j\ne q) \quad \text{and} \quad J'_{u,j}=J_j\ (j\notin\supp a).
\]
\end{lemma}
\begin{proof}
Since the construction below only involves the coordinates in $\supp a$, we may assume that
$\supp a=\{1,\ldots,d\}$ (on the other coordinates, $D$ is the identity and $J'_{u,j}=J_j$). In
particular $d\le|a|_1\le R$. Reordering, we may also assume that $\#J_1\ge\cdots\ge\#J_d$, and so $q=1$.

Let us first briefly outline the argument. We want new coordinates in which $a\cdot x$ is a
coordinate whose values on $\cP(W)$ lie in an interval of $H$ integers, $Q$ is still diagonal, and the other coordinates have ranges comparable to the old sides. Over $\mathbb Q$ such
coordinates come from Gram--Schmidt: we first construct $T\in\mathrm{GL}_d(\mathbb Q)$ such that, with
$z=Tx$, $Q$ is diagonal in $z$, $z_1=a\cdot x$ varies on $\cP(W)$ by at most $RH$, and $z_j$ varies by at
most $(2R+1)\#J_j$ for $j\ge2$. The substitution $x=T^{-1}z$ then has the required properties, except
that it is not integral and $z_1$ varies by up to $RH$ rather than by less than $H$. We fix both by passing to
$D=sT^{-1}$ for a suitable integer $s>R$: the substitutions $x=u+Dy$ divide all ranges by $s$, and the
cosets $u+D\z^d$ partition $\z^d$.

\emph{A rational change of variables.} For $\xi,\eta\in\mathbb Q^d$ put
$\langle\xi,\eta\rangle=\sum_i\xi_i\eta_i/\gamma_i$. If $r_1,{\ldots},r_d\in\mathbb Q^d\setminus\{0\}$ are
orthogonal for $\langle\cdot,\cdot\rangle$, then, with $G=\operatorname{diag}(\gamma_i)$,
\begin{equation}\label{eq:parseval}
 Q(x)=\langle Gx,Gx\rangle=\sum_{j=1}^d\frac{\langle Gx,r_j\rangle^2}{\langle r_j,r_j\rangle}
 =\sum_{j=1}^d\frac{(r_j\cdot x)^2}{\langle r_j,r_j\rangle}.
\end{equation}
Applying Gram--Schmidt to $a,e_d,e_{d-1},{\ldots},e_2$, we find $r_1=a$ and
\[
 r_j=e_j-\frac{a_j}{\gamma_j\tau_j}\,(a_1,{\ldots},a_j,0,{\ldots},0) \quad \text{with} \quad
 \tau_j=\sum_{i\le j}\frac{a_i^2}{\gamma_i}\qquad(2\le j\le d).
\]
Indeed, $\langle a,r_j\rangle=0$, and for $2\le j<k$, since $r_j$ is supported on the first $j$
coordinates,
\[
 \langle r_j,r_k\rangle=\langle r_j,e_k\rangle-\frac{a_k}{\gamma_k\tau_k}\langle r_j,a\rangle=0.
\]
Let $T\in\mathrm{GL}_d(\mathbb Q)$ have rows $r_1,{\ldots},r_d$, and put $z=Tx$, so that $z_j=r_j\cdot x$.
Then $z_1=a\cdot x$, and $Q$ is diagonal in $z$ by \eqref{eq:parseval}. On $\cP(W)$, $z_1$ varies by at most $RH$. For $j\ge2$,
\begin{equation}\label{eq:zj}
 z_j=x_j-\frac{a_j}{\gamma_j\tau_j}\,a\cdot x+\frac{a_j}{\gamma_j\tau_j}\sum_{\ell>j}a_\ell x_\ell,
 \qquad\text{where }\ \frac{|a_j|}{\gamma_j\tau_j}\le1\ \text{ as }\ \tau_j\ge\frac{a_j^2}{\gamma_j}.
\end{equation}
Since $\#J_\ell\le\#J_j$ for $\ell>j$, the three terms on the right-hand side of \eqref{eq:zj} vary on
$\cP(W)$ by at most $\#J_j-1$, $RH$ and $R(\#J_j-1)$, respectively. As $H\le\#J_j$, it follows that
$z_j$ varies by at most $(2R+1)\#J_j$.

\emph{Integral version.} As $T^{-1}$ does not preserve $\z^d$, we rescale it and split into cosets.
Let $s>R$ be the least integer such that $D=sT^{-1}$ is integral, and let $U=D[0,1)^d\cap\z^d$. The
cosets $u+D\z^d$, $u\in U$, partition $\z^d$ (see  \cite[Proposition~1.3.10 and Corollary~1.3.27]{HNP}). For
$u\in U$ let $W'_u=\{(D^{\mathsf T}w,\,b+w\cdot u):(w,b)\in W\}$, so that
$\cP(W'_u)=\{y\in\z^d:u+Dy\in\cP(W)\}$. Hence the sets $u+D\cP(W'_u)$, with $u\in U$, partition $\cP(W)$.

\emph{The new form.} Let $Q'(y)=Q(Dy)$ for $y\in\mathbb Q^d$. By \eqref{eq:parseval} we obtain
\[
 Q'(y)=\sum_{j=1}^d\frac{(r_j\cdot Dy)^2}{\langle r_j,r_j\rangle}
 =s^2\sum_{j=1}^d\frac{y_j^2}{\langle r_j,r_j\rangle},
\]
since $r_j\cdot Dy$ is the $j$-th entry of $TDy=sy$, as $T$ has rows $r_j$. So $Q'$ is diagonal. Taking $y=e_j$ in the display, we find that the \(j\)th coefficient of \(Q'\) is $s^2/\langle r_j,r_j\rangle=Q'(e_j)=Q(De_j)$; a positive integer since $De_j$ is a nonzero integer
vector.

\emph{The new boxes.} If $y\in\cP(W'_u)$, then $x=u+Dy\in\cP(W)$ and $Tx=Tu+sy$, so
$y_j=(r_j\cdot x-(Tu)_j)/s$. Hence on $\cP(W'_u)$ each $y_j$ varies by at most $1/s$ times the variation
of $z_j=r_j\cdot x$ on $\cP(W)$. The coordinate $y_1$ varies by at most $RH/s<H$, and $y_j$ by at most $(2R+1)\#J_j/s<3\#J_j$
for $j\ge2$. Hence $\cP(W'_u)$ lies in a box $J'_u$ with $\#J'_{u,1}=H$ and
$\#J_j\le\#J'_{u,j}\le3\#J_j$ for $j\ge2$.

\emph{Bounds.} Since $d\le|a|_1\le R$ and $\gamma_j\le R$, there are $\ll_R1$ choices of $a$ and
$(\gamma_j)$, and $D$ depends only on these, hence $|D|_\infty$, $\#U=|\det D|$, $\rho(W'_u)$ and the
coefficients of $Q'$ are $\ll_R1$. As $(W,J,Q,R)$ has dimension at most $R$ and
$\cP(W'_u)\subseteq J'_u$, each $(W'_u,J'_u,Q',R')$ is a piece for some $R'\ll_R1$, and the transitions
$x=u+Dy$ to these pieces form a decomposition of $(W,J,Q,R)$.
\end{proof}

\subsection{Decomposition of pieces}\label{sec:pieces}

\begin{proposition}\label{prop:pieces}
For all $d,n_{\mathrm{long}},n_{\mathrm{mix}},A$ there is $C=C(d,n_{\mathrm{long}},n_{\mathrm{mix}},A)$ such that every piece $(W,J,Q,A)$ of
index $(d,n_{\mathrm{long}},n_{\mathrm{mix}})$ has a decomposition by at most $C$ transitions
$x=u_m+\Phi_my$ of size at most $C$ to pieces $(W_m,J_m,Q_m,A_m)$ with $A_m\le C$,
$\vol(J_m)\le C\vol(J)$ and $J_m$ $C$-balanced with respect to $W_m$.
\end{proposition}

\begin{proof}
We proceed by induction on $(d, n_{\text{long}}, n_{\text{mix}})$ ordered lexicographically. 
Thus given a piece $(W, J,  Q, A)$ of index $\rho = (d, n_{\text{long}}, n_{\text{mix}})$, we 
assume that the statement holds
for all $A \geq 1$ and all pieces $(W', J', Q', A')$ of index $< \rho$.  

We organize the proof according to the following three possibilities. 

\textbf{Case $n_{\text{long}} = 0$.} In this case all sides of $J$ lie between $H := \min_i \#J_i$ and
$A H$ and $\cP(W) \subset J$. Therefore $J$ is $A$-balanced with respect to $W$, and the identity transition
from $(W,J,Q,A)$ to itself is the required decomposition.

\textbf{Case $n_{\text{long}} > 0$ and $n_{\text{mix}} = 0$.} 
We can without loss of generality assume that 
$B = \{1, \ldots, t\}$ and $B^c = \{t + 1, \ldots, d\}$ for
some $1 \leq t < d$.
Since $n_{\text{mix}} = 0$, for every $(w,b) \in W$ either $\supp w \subset B$
or $\supp w \subset B^c$. Thus,
$\cP(W)=\cP(W_B)\times\cP(W_{B^c})$, where for $E \subseteq \{1, \ldots, d\}$, we set
$$W_E=\bigl\{((w_i)_{i\in E},\,b):(w,b)\in W,\ \supp w\subseteq E\bigr\}\subset\Z^E\times\Z.$$ 
{Likewise, let} 
$$
J_E := \prod_{i \in E} J_i {\quad \text{and} \quad} Q_E(x) := \sum_{j \in E} \gamma_j x_j^2.
$$
Then, $J=J_B\times J_{B^c}$ and
$Q=Q_B+Q_{B^c}$. Both $(W_E, J_E, Q_E, A)$ with $E \in \{B, B^c\}$ are pieces of dimension 
less than $d$, so the induction hypothesis gives decompositions of them into pieces
$(W_m,J_m,Q_m,A_m)$, $1\le m\le M$, and $(W'_n,J'_n,Q'_n,A'_n)$, $1\le n\le N$, respectively. Since
$\cP(W)=\cP(W_B)\times\cP(W_{B^c})$ and $Q=Q_B\oplus Q_{B^c}$, by Lemma~\ref{lem:compose}(2) these
combine into a decomposition of $(W,J,Q,A)$ into the pieces
$(W_m\oplus W'_n,J_m\times J'_n,Q_m\oplus Q'_n,A_m+A'_n)$.
Finally, let 
$$
R := \max_{\substack{d'<d \\ n'_{\text{long}}<d' \\ n'_{\text{mix}} \leq A}} C(d',n'_{\text{long}},n'_{\text{mix}},A)\ll_A1.
$$
By induction each $J_m$ and $J'_n$ is $R$-balanced with respect to $W_m$
and $W'_n$ respectively. Moreover since for each $(w,b) \in W_m \oplus W'_n$ either 
$\supp w \subseteq B$ or $\supp w \subseteq B^c$, the graph $\Gamma(W_m \oplus W'_n)$ 
is the disjoint union of $\Gamma(W_m)$ on $B$ and $\Gamma(W'_n)$ on $B^c$. 
Therefore the box $J_m \times J'_n$ is also $R$-balanced with respect to 
$W_m \oplus W'_n$, and $\vol(J_m\times J'_n)\le R^2\vol(J)$. 
Hence this is the required decomposition of $(W,J,Q,A)$.

\textbf{Case $n_{\text{long}} > 0$ and $n_{\text{mix}} > 0$.} 
Let $(w,b) \in W$ correspond to a mixed inequality.
Let $H := \min_i \# J_i$ and $q$ the index such that 
$$
\max_{i \in \supp w_{B^c}} \# J_i = \# J_q.
$$
By Lemma~\ref{lem:split}(1), splitting along $(w,b)$ decomposes $(W,J,Q,A)$ into
$(W^+,J,Q,A+2d)$ and $(W^-,J,Q,A+1)$, and by Lemma~\ref{lem:split}(2) the induction hypothesis gives
a decomposition of $(W^+,J,Q,A+2d)$ as required. It remains to obtain a decomposition of
$(W^-,J,Q,A+1)$.

On $\cP(W^-)$ the inner product $w_{B^c}\cdot x$ varies by less than $A^2H$, so
Lemma~\ref{lem:short-coordinate}, applied with $a=w_{B^c}$ to the piece $(W^-,J,Q,(A+1)^2)$, gives a decomposition of it by $\#U\ll_A1$ transitions $x=u+Dy$ to the pieces $(W'_u,J'_u,Q',R')$,  with $u\in U$, where $R'\ll_A1${. We} may
assume that $R'\ge A$.

By Lemma \ref{lem:short-coordinate} for each $u \in U$ in this decomposition 
the box $J'_u$ has the property that $\# J'_{u,q} = H$ and $H \leq \# J'_{u,i}$ for all $i$
and the coordinates with $i \in B$ remain of the same length, i.e. $\# J'_{u,i} = \# J_i$. 
Thus compared to $(W, J, Q, A)$, $n_{\text{long}}$ has decreased in each $(W'_u,J'_u,Q',R')$, 
since $q$, which was previously long in $(W,J, Q, A)$, 
became short in $(W'_u,J'_u,Q',R')$ while no other index that was short 
in $(W, J, Q, A)$ became long in $(W'_u,J'_u,Q',R')$. 
Thus the piece $(W'_u,J'_u,Q',R')$ has smaller index than $(W, J, Q, A)$, and the
induction hypothesis gives a decomposition of it with constant $\ll_A1$. By the side bounds of
Lemma~\ref{lem:short-coordinate} and $H\le\#J_q$,
\[
 \vol(J'_u)=H\prod_{j\ne q}\#J'_{u,j}\le3^{d-1}\frac{H}{\#J_q}\vol(J)\ll_A\vol(J).
\]
Composing (Lemma~\ref{lem:compose}) the decompositions of $(W^+,J,Q,A+2d)$ and of the
$(W'_u,J'_u,Q',R')$ with the transitions above gives the required decomposition of $(W,J,Q,A)$.
\end{proof}

\subsection{Proof of Theorem~\ref{thm:decomp}}\label{sec:proof-decomp}
\begin{proof}
Let $A=\max\{k,r,V,h\}$ and $Q(x)=\sum_{j=1}^kc_jx_j^2$. Then $(W,I,Q,A)$ is a piece, and
Proposition~\ref{prop:pieces} gives a decomposition of it into pieces
$(W_m,J_m,Q_m,A_m)$ by transitions $x=u_m+\Phi_my$ (this is \eqref{eq:decomposition}). The other
conclusions follow from the definition of a piece.
Since $0\le n_{\mathrm{long}}<k$ and $0\le n_{\mathrm{mix}}\le A$, the constant depends only on $k$ and
$A$, hence on $k,r,V,h$.

If $(\pm w_0,\pm b_0)\in W$, then $w_0\cdot x$ is constant on $P$, and Lemma~\ref{lem:short-coordinate}
(applied to the piece $(W,I,Q,A)$, with $a=w_0$ and $H=1$) gives a decomposition of $(W,I,Q,A)$ into
$\ll_A1$ pieces $(W'_u,J'_u,Q',R')$ with $R'\ll_A1$ and $\vol(J'_u)\le3^{k-1}\vol(I)/\len(w_0;I)$, as
$\#J'_{u,q}=1$ and $\#I_q=\len(w_0;I)$. Composing (Lemma~\ref{lem:compose}) with the decompositions of
these pieces given by Proposition~\ref{prop:pieces} gives \eqref{eq:decomposition-section}.
\end{proof}

\subsection{Proof of Theorem~\ref{thm:box}}\label{sec:proofbox}

\begin{proof}[Proof of Theorem~\ref{thm:box}]
Inequalities $(0,b)\in L$ are either vacuous or make the sum empty, so we may assume that $v\ne0$ for
all $(v,b)\in L$. Let $W$ consist of $L$ together with the $2k$ inequalities defining $I$, so that
$P:=\cP(L)\cap I=\cP(W)$, and assume $P\ne\emptyset$. Let $I_0\subseteq I$ be the smallest box
containing $P$. Then $\vol(I_0)\le\vol(I)$ and $\vol(I_0)/\len(w_0;I_0)\le\vol(I)/\len(w_0;I)$, since
$\vol(J)/\len(w_0;J)$ is the product of the sides of $J$ other than a longest one in $\supp w_0$.

Theorem~\ref{thm:decomp}, applied to $W$, gives the decomposition \eqref{eq:decomposition} of $P$ with
$R=R(k,r+2k,V,h)$. Substituting $n=u_m+\Phi_my$ on the $m$-th piece gives
\[
 \alpha\sum_{j=1}^k c_jn_j^2+\theta\cdot n
 =\alpha\sum_{j=1}^k c'_{m,j}y_j^2+\theta_m\cdot y+\kappa_m
\]
with $\theta_m\in\RR^k$ and $\kappa_m\in\RR$, where the $c'_{m,j}=\sum_ic_i\Phi_{m,ij}^2$ are the diagonal
entries of the diagonal matrix $\Phi_m^{\mathsf T}\operatorname{diag}(c_j)\Phi_m$. They are positive
integers, as $\Phi_m$ is non-singular, and at most $R$. Hence
\[
 \Bigl|\sum_{n\in P}e\Bigl(\alpha\sum_{j=1}^kc_jn_j^2+\theta\cdot n\Bigr)\Bigr|
 \le\sum_{m=1}^M\Bigl|\sum_{y\in\cP(W_m)\cap J_m}
 e\Bigl(\alpha\sum_{j=1}^kc'_{m,j}y_j^2+\theta_m\cdot y\Bigr)\Bigr|.
\]

Fix $m$, and let $C_1,{\ldots},C_s$ be the components of $\Gamma(W_m)$. As each $(w,b)\in W_m$ has
$\supp w$ inside a single $C_t$, using the notation of the proof of
Proposition~\ref{prop:pieces}, that is, for $E\subseteq\{1,{\ldots},k\}$, $W_E$ consists of the $((w_i)_{i\in E},b)$
with $(w,b)\in W$ and $\supp w\subseteq E$, and $J_E=\prod_{i\in E}J_i$, we have
\[
 \cP(W_m)\cap J_m=\prod_{t=1}^s\bigl(\cP\bigl((W_m)_{C_t}\bigr)\cap(J_m)_{C_t}\bigr).
\]
The phase is a sum over $j$ of terms depending only on $y_j$, so the inner sum splits accordingly into a product of
$s$ sums. By Theorem~\ref{thm:decomp}, each pair $\bigl((W_m)_{C_t},(J_m)_{C_t}\bigr)$ is $R$-balanced
in the sense of Definition~\ref{def:balanced}, so Theorem~\ref{thm:general}, with $R$ replaced by
$\max(R,a)$, bounds the $t$-th sum by $\ll\vol\bigl((J_m)_{C_t}\bigr)^{1/2}$. Taking the product over
$t$, the $m$-th inner sum is $\ll\vol(J_m)^{1/2}$.

Summing over the $M\le R$ pieces and using \eqref{eq:boxed-conclusion}, which gives
$\vol(J_m)\le R\vol(I_0)\le R\vol(I)$, proves the first bound. If $\pm(w_0,b_0)\in L$, then
\eqref{eq:decomposition-section} gives instead
$\vol(J_m)\le R\vol(I_0)/\len(w_0;I_0)\le R\vol(I)/\len(w_0;I)=R\,\sect(w_0;I)$. All
implied constants depend only on $k,r,V,h,a$.
\end{proof}
\section{Equidecompositions: Proof of Proposition \ref{prop:UNmap}}
\label{sec:explicitEquidecomposition}

This {section} proves Proposition \ref{prop:UNmap}, giving two distinct construction methods  for the map $U_N$. 
It uses the notation and definitions written down prior to Proposition \ref{prop:UNmap}, which we recommend {that the reader recall}.

\subsection{Unimodular Triangulations (the {implicit} method)}

We say that $P\subset \RR^k$ is an \emph{integral convex polytope} if $P = \widebar{\text{conv}}(a_1,a_2,{\ldots},a_n)$ is the closed convex hull of a finite number of integer lattice points  $a_j \in \z^k$. An \emph{$m$-simplex} $P$, for $0\leq m\leq k$, 
is the convex hull of $m+1$ affinely independent points
$a_0,{\ldots},a_m \in \Z^k$ over $\RR$. It is \emph{unimodular} if the matrix
$
A = (a_1-a_0, \ldots, a_m-a_0)
$
has columns forming a $\Z$-basis of $A\RR^m\cap\Z^k$. When $m=k$, this is equivalent to $\det(A)=\pm1$, and to the Lebesgue volume of $P$ being $1/k!$ (see \cite[Section~2.1 and Definition~2.3.2]{HNP}). An \emph{open polytope} (respectively, \emph{open simplex}) is the relative interior of a polytope (respectively, simplex). A \emph{unimodular triangulation} of $P$ is a partition $P=\bigsqcup_{i=1}^{n_0}T_i$, where $T_i$ are open unimodular simplices{ whose closures intersect in common faces or are disjoint}. {If $P$ admits a unimodular triangulation, then} in any such triangulation the number $f_m$ of unimodular $m$-simplices used is the same. The \emph{Ehrhart polynomial} associated with a polytope $P$ is $p(t)=\#(t P \cap \Z^k)$ for $t\in \Z_{\geq 0}$. A classical result {of} Ehrhart \cite{Ehr} shows that $p(t)$ is indeed a polynomial and{, for any unimodular triangulation of $P$,}
\[
 |tP\cap\mathbb Z^k|=\sum_{m=0}^k f_m\binom{t-1}{m}{\qquad(t\geq1)}.
\]
The following argument is partially contained in {\cite[Proposition~4.3]{HM08}}, except that we specialize it to the convex {hulls} of the sets
\begin{align*}
\cS_N&:=\left\{j\in \Z_{\ge0}^k:
  j_1+\cdots+j_k\le N\right\} \quad \text{and} \quad
\B_N:=\prod_{s=1}^k \bigg\{0,1,\ldots, \frac{N}{k+1-s}\bigg\}.
\end{align*}

\begin{proof}[Proof of Proposition \ref{prop:UNmap} via {unimodular} triangulations]
{For $N=0$, set $U_0(0)=0$. Assume $N\geq1$.} Consider the integral convex polytopes $S=\{x\in \RR_{\geq 0}^k : x_1+\cdots{+} x_k \leq \lcmk\}$ and $B = \prod_{s=1}^k [0,\lcmk/(k+1-s)] $, where $\lcmk=\text{lcm}(1,2,\ldots,k)$. {The facet normals of $S$ are the negative coordinate vectors and $(1,\ldots,1)$, while those of $B$ are the signed coordinate vectors. Every $k\times k$ determinant formed from either collection is $0$ or $\pm1$, so $S$ and $B$ admit unimodular triangulations by \cite[Theorem~6.3.3]{HNP}.} It is clear that for every integer $t\geq1$,
\begin{align*}
|t {S} \cap\Z^k|
=\binom{t \lcmk+k}{k}
=\prod_{s=1}^k\frac{t\lcmk+k+1-s}{k+1-s}
=|tB\cap\Z^k|.
\end{align*}
Thus ${S}$ and $B$ have the same Ehrhart polynomial. In particular, there are integers $f_0,f_1,\ldots,{f_k}\geq 0$ and collections $\{T_{m,1},\ldots, T_{m,f_m}\}$ and $\{T_{m,1}',\ldots, T_{m,f_m}'\}$ of open unimodular $m$-simplices, for $m=0,1,\ldots,k$, such that 
\[
S=\bigsqcup_{m=0}^{k} \bigsqcup_{i=1}^{f_m} T_{m,i} \quad \text{and} \quad B=\bigsqcup_{m=0}^{k} \bigsqcup_{i=1}^{f_m} T_{m,i}'.
\]
Any two open unimodular simplices {of the same dimension} are ${\aff_k(\Z)}$-equivalent{. That} is, there are
affine unimodular maps $U_{m,i}(x)=A_{m,i}x+b_{m,i}$ with $A_{m,i}\in\GL_k(\Z)$ and $b_{m,i}\in\Z^k$, 
such that $T_{m,i}'=U_{m,i}(T_{m,i})$ for all $m$ and $i$. Observe now that  when $\lcmk|N$ we have
\[
\tfrac{N}{\lcmk}S=\bigsqcup_{m=0}^{k} \bigsqcup_{i=1}^{f_m} \tfrac{N}{\lcmk} T_{m,i} \quad \text{and} \quad {\tfrac{N}{\lcmk}}B=\bigsqcup_{m=0}^{k} \bigsqcup_{i=1}^{f_m} U_{m,i,N}(\tfrac{N}{\lcmk} T_{m,i}),
\]
where $U_{m,i,N}(x)=A_{m,i}x+\frac{N}{\lcmk}b_{m,i}$. Intersecting with $\Z^k$, we obtain
\[
\cS_N =\bigsqcup_{m=0}^{k} \bigsqcup_{i=1}^{f_m} \big((\tfrac{N}{\lcmk} T_{m,i}) \cap \Z^k\big) \quad \text{and}\quad \B_N =\bigsqcup_{m=0}^{k} \bigsqcup_{i=1}^{f_m} U_{m,i,N}\big((\tfrac{N}{\lcmk} T_{m,i}) \cap \Z^k\big), 
\]
where we used that $U_{m,i,N}(\tfrac{N}{\lcmk} T_{m,i}) \cap \Z^k = U_{m,i,N}\big((\tfrac{N}{\lcmk} T_{m,i}) \cap \Z^k\big) $ because each $U_{m,i,N} \in \aff_k(\Z)$. We then define $U_N(x):=U_{m,i,N}(x)$ if $x\in \cS_N \cap \frac{N}{\lcmk} T_{m,i}$. It is easy to see that $U_{N}: \cS_N \to \B_N$ {is a} piecewise $\aff_k(\Z)$ bijection, where the domains are the $\Z$-polytopes $(\frac{N}{\lcmk} T_{m,i} ) \cap \Z^k$ and thus we have an $\aff_k(\Z)$-equidecomposition of $\cS_N$ and $\B_N$. The complexity of this equidecomposition is independent of $N$ since we have $O_k(1)$ regions, each region is determined by $\leq 2k+1$ affine inequalities, we have $O_k(1)$ linear maps{, and all their defining normals and matrix entries depend only on $k$}. This finishes the proof{ of the $O_k(1)$ bound{.}}
\end{proof}

\subsection{The direct fold in every dimension (the constructive method)}
The map $U_N$ will be built by a recursive construction on the dimension of the ambient space $k$.
{Let's} then {redefine} $U_{k,N}=U_N$, $\cS_{k,N}=\cS_N$ and $\B_{k,N}=\B_N$. Let $\lcmk \mid N$, where $\lcmk = \text{lcm}(1,2,\ldots,k)$. First{, for $k\geq2$,} we construct an $\aff_k(\Z)$-equidecomposition of $\cS_{k,N}$ and $\{0,1,\ldots,N/k\} \times \cS_{k-1,N}$ of complexity bounded by ${2k+1}$. This is done in Proposition \ref{prop:fold}. The benefit here is that everything is explicit and analytic. This produces a piecewise $\aff_k(\Z)$ bijection $\Phi_{k,N}: \cS_{k,N}\to \{0,1,\ldots,N/k\} \times \cS_{k-1,N}$ defined in \eqref{eq:fold}. We then simply set
\begin{align}\label{def:UN}
U_{k,N}(x) = (l_1,  U_{k-1,N}(l_2,\ldots,l_k) ) \quad \text{if} \quad \Phi_{k,N}(x) = (l_1,l_2,\ldots,l_k).
\end{align}
One can then  deduce the explicit formula
\begin{align*}
U_{k,N}
=
\Bigl(
\pi_{k-2}^{\mathsf T}\pi_{k-2}
+(\pi^{k-2})^{\mathsf T}\circ\Phi_{2,N}\circ\pi^{k-2}
\Bigr)\circ\cdots\circ
\Bigl(
\pi_1^{\mathsf T}\pi_1
+(\pi^1)^{\mathsf T}\circ\Phi_{k-1,N}\circ\pi^1
\Bigr)
\circ\Phi_{k,N},
\end{align*}
where $\pi_j(x)=(x_1,\ldots,x_j)$ and $\pi^j(x)=(x_{j+1},\ldots,x_k)$,  with their right inverses $\pi_j^{\mathsf T}(u)=(u,0,\ldots,0)$
and $(\pi^j)^{\mathsf T}(v)=(0,\ldots,0,v)$. In particular, we have $U_{1,N}=\mathrm{Id}$. For
$2\mid N$ and $x\in\cS_{2,N}$ we have
\[
U_{2,N}(x)=
\begin{cases}
\bigl(x_1,x_1+x_2\bigr),
& \text{if }x_1\leq N/2,\\
\bigl(x_1+x_2-N/2,x_1-N/2-1\bigr),
& \text{if }x_1>N/2.
\end{cases}
\]
For $6\mid N$ and $x\in\cS_{3,N}$ we have
\[
U_{3,N}(x)=
\begin{cases}
\bigl(x_1,x_1+x_2,x_1+x_2+x_3\bigr),
& \text{if }
  \begin{aligned}[t]
  x_1&\leq N/3,\, x_1+x_2\leq N/2,
  \end{aligned}\\[0mm]
\bigl(x_1,x_1+x_2+x_3-N/2,x_1+x_2-N/2-1\bigr),
& \text{if }
  \begin{aligned}[t]
  x_1&\leq N/3,\, x_1+x_2>N/2,
  \end{aligned}\\[0mm]
\bigl(x_1+x_2-N/3,x_1-N/3-1,x_1+x_2+x_3\bigr),
& \text{if }
  \begin{aligned}[t]
  x_1&>N/3,\, x_1+x_2\leq 2N/3,
  \end{aligned}\\[0mm]
\bigl(x_1+x_2+x_3-2N/3,x_3,x_1+x_3-N/3-1\bigr),
& \text{if }
  \begin{aligned}[t]
  x_1&>N/3,\, x_1+x_2>2N/3.
  \end{aligned}
\end{cases}
\]

\begin{proof}[Proof of Proposition \ref{prop:UNmap} via the folding method]
{For $N=0$, set $U_{k,0}(0)=0$. Assume $N\geq1$. We define $U_{1,N}(x)=x$. This has complexity $R_1=2$.} {For $k\geq2$, define} $U_{k,N}$ recursively as in \eqref{def:UN}, where $\Phi_{k,N}$ is defined in \eqref{eq:fold}. From \eqref{eq:sourcePieces}, \eqref{eq:fold} and \eqref{eq:targetPieces} we deduce the following schematic picture:
\[
\cS_{k,N} = \bigsqcup_{m=1}^{k} \C_{k,N,m} \ \xrightarrow{{\Phi_{k,N}}} \ \{0,1,\ldots,N/k\} \times \cS_{k-1,N} = \bigsqcup_{m=1}^{{k}} (A_{k,N,m}  \C_{k,N,m} + a_{k,N,m} ),
\]
where $\D_{k,N,m}={A_{k,N,m}  \C_{k,N,m} + a_{k,N,m}}$ and the $\aff_k(\Z)$-maps ${A_{k,N,m} x + a_{k,N,m}}$ are {defined} by \eqref{eq:fold}. We see that each matrix ${A_{k,N,m}}$ only has $\{0,1\}$-{entries} and that each $\Z$-polytope ${\C_{k,N,m}}$ is defined by {the affine inequalities in \eqref{eq:sourcePieces}, together with the coordinate inequalities $x_j\geq0$ and $x_1+\cdots+x_k\leq N$. Their facet normals have $\ell^1$-norm at most $k$.} By Proposition \ref{prop:fold} we also know that ${A_{k,N,m}} \in GL_{k}(\Z)$. Now assume by induction that, for $k\geq 2$, we have already {built} the map $U_{k-1,N}$, induced by the following $\aff_{k-1}(\Z)$-equidecomposition of $\cS_{k-1,N}$ and $\B_{k-1,N}$:
\[
\cS_{k-1,N} = \bigsqcup_{i=1}^{n_{k-1}} \E_{k-1,N,i} \ \xrightarrow{U_{k-1,N}} \ \B_{k-1,N} = \bigsqcup_{i=1}^{n_{k-1}} (F_{k-1,N,i}  \E_{k-1,N,i} + f_{k-1,N,i} ),
\]
{where $\E_{k-1,N,i}=\cP(L_{k-1,N,i})$.}
Assume that $U_{k-1,N}$ has complexity $ \leq R_{k-1}$, that is, $n_{k-1}\leq R_{k-1}$, {$\#L_{k-1,N,i}\leq R_{k-1}$,} {$|v|_1 \leq R_{k-1}$} for each  $(v,b) \in {L_{k-1,N,i}}$ and ${|F_{k-1,N,i}|_{\infty}} \leq R_{k-1}$. By the definition \eqref{def:UN}, we obtain the following decomposition of $\cS_{k,N}$ for the map $U_{k,N}$:
\[
 \cS_{k,N} = \bigsqcup_{m=1}^{k} \bigsqcup_{i=1}^{n_{k-1}} \C_{k,N,m} \cap \cP\Bigl(\big\{ \big(A_{k,N,m}^T (0,v),{b+(0,v)\cdot a_{k,N,m}}\big) \, : \, (v,b) \in L_{k-1,N,i}\big\}\Bigr).
\]
{Indeed, the additional inequality is $(0,v)\cdot(A_{k,N,m}x+a_{k,N,m})+b\geq0$.}
Momentarily, we let $\C_{k,N,m,i}$ denote the pieces above. Note that {a priori} some of these pieces might be empty, but nevertheless they form a $\Z$-polytope partition of $\cS_{k,N}$. On each non-empty piece $\C_{k,N,m,i}$, the map $U_{k,N}$ is given by
\[
x\in\RR^k\mapsto
\Bigl(\pi_1(A_{k,N,m}x+a_{k,N,m}),\,
F_{k-1,N,i}\pi^1(A_{k,N,m}x+a_{k,N,m})+f_{k-1,N,i}\Bigr),
\]
where $\pi_1:\RR^k \to \RR$ is the projection {onto} the first component and $\pi^1:\RR^k \to \RR^{k-1}$ is {the} projection {onto} the last $k-1$ components. These maps are easily invertible, with similar expressions, using the inverses $A_{k,N,m}^{-1}$ and $F_{k-1,N,i}^{-1}$, which shows they belong to $\aff_k(\Z)$.
We conclude that $U_{k,N}$ induces {an} $\aff_k(\Z)$-equidecomposition of $\cS_{k,N}$ and $\B_{k,N}$. It remains to compute its complexity $R_k$. Let $n_k $ be the number of non-empty $\C_{k,N,m,i}$. We see immediately that $n_k \leq k n_{k-1} \leq k R_{k-1}$. Moreover,
\begin{align*}
{\left|\begin{pmatrix}\pi_1\circ A_{k,N,m}\\F_{k-1,N,i}\circ\pi^1\circ A_{k,N,m}\end{pmatrix}\right|_\infty}
&\leq | A_{k,N,m} |_\infty + (k-1)  | F_{k-1,N,i}|_\infty | A_{k,N,m} |_\infty \\
& \leq 1 + (k-1)R_{k-1} \leq k R_{k-1}.
\end{align*}
Finally, the affine {inequalities} defining each $\C_{k,N,m,i}$ have normals with $\ell^1$-norm bounded by
\[
\max_{(v,b) \in L_{k-1,N,i}}\{k, |A_{k,N,m}^T (0,v)|_1\} \leq \max_{(v,b) \in L_{k-1,N,i}}\{k, {k}\,|A_{k,N,m}|_\infty |v|_1\} \leq \max\{k,{kR_{k-1}}\} \leq kR_{k-1}.
\]
{Each piece is defined by at most $2k+1+R_{k-1}$ inequalities. We may therefore take $R_k=kR_{k-1}+2k+1$, which depends only on $k$. This finishes the proof.}
\end{proof}

\begin{figure}[H]
  \centering
     \includegraphics[scale=.3]{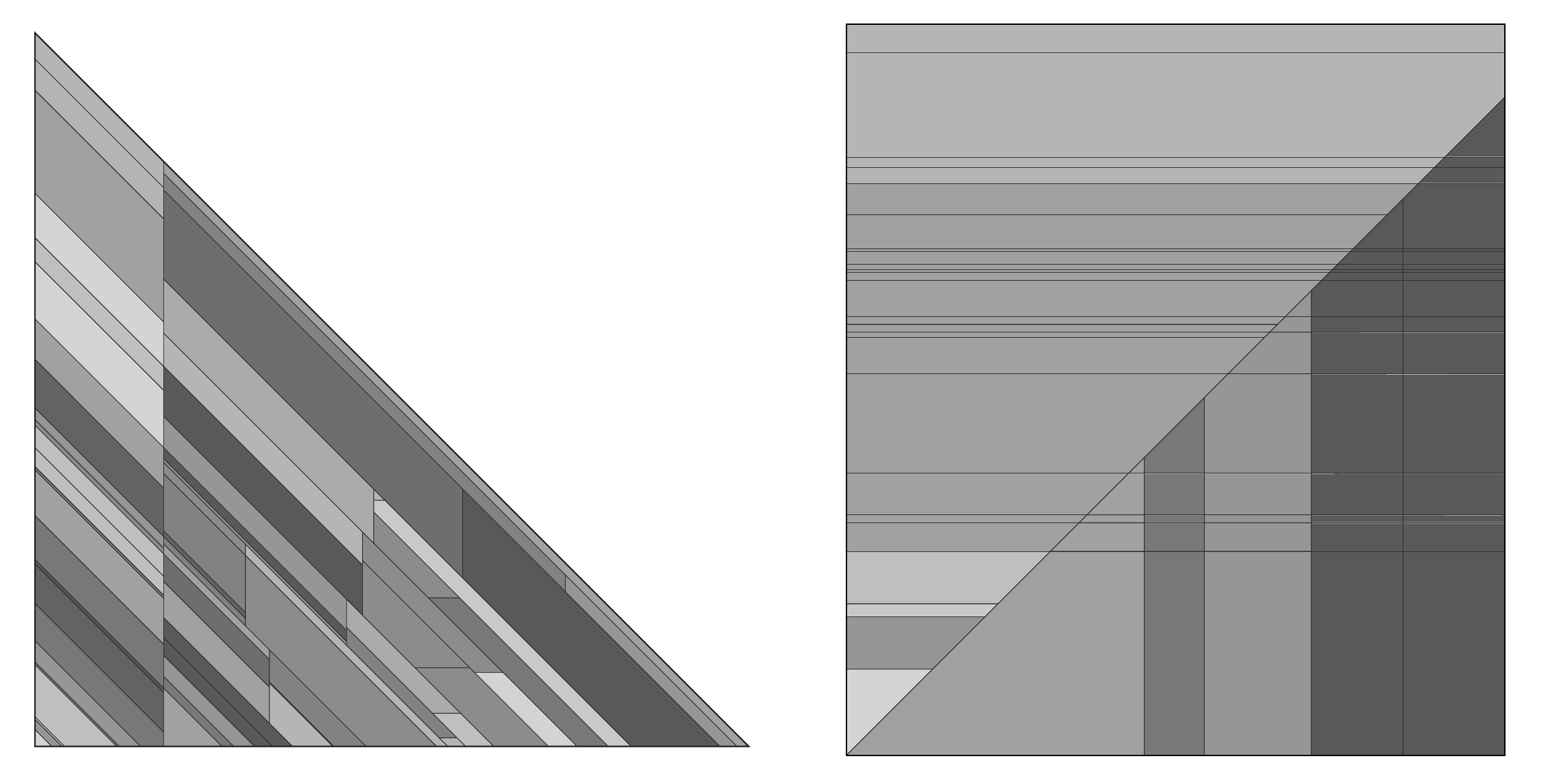} 
  \caption{For $N=2520$, the left picture shows the equidecomposition of
$\cS_{10,N}$ sliced along the plane
$(x_3,\ldots,x_{10})=(108,92,162,89,146,78,219,229)$.
On the right we show its image under $U_N$, projected onto the
$(y_1,y_2)$-plane of $\B_{10,N}$.}
  \label{fig:U10}
\end{figure}

We now construct  an explicit $\aff_k(\Z)$-equidecomposition of $\cS_{k,N}$ and $\{0,1,\ldots,N/k\} \times \cS_{k-1,N}$.
Set $\Phi_{1,N} = \text{Id}$.  For $k\geq2${, $N\geq0$ with $k\mid N$,} and ${1\leq m\leq k}$, define
\begin{align}
\C_{k,N,m}=\left\{x\in{\cS_{k,N}}:
 \sum_{\nu=1}^{i}x_\nu>\frac{iN}{k}\ (1\leq i<m),
 \quad
 \sum_{\nu=1}^{m}x_\nu\leq\frac{mN}{k}\right\}.
\label{eq:sourcePieces}
\end{align}
These are disjoint $\Z$-polytopes {that produce a partition of}
${\cS_{k,N}}$. {In what follows, coordinate expressions or constraints containing {out-of-bounds} variables should be omitted in their entirety. Coordinate blocks with empty index ranges should also be omitted.} For $x\in\C_{k,N,m}$, define the exact map directly by
\begin{align}
\Phi_{k,N}(x)
&=
\begin{cases}
\big(x_1,x_1+x_2,x_3,\ldots,x_k\big),
& \text{if }m=1,\\[1mm]
\begin{aligned}[t]
&\left(\displaystyle\sum_{\nu=1}^{m}x_\nu
-\frac{(m-1)N}{k},
x_m,x_{m-1},\ldots,x_3,\right.\\[-1mm]
&\quad \quad \left.x_1-\frac Nk-1,
x_2+x_{m+1}+\frac Nk+1,x_{m+2},\ldots,x_k\right),
\end{aligned}
& \text{if }2\leq m \leq k.
\end{cases}
\label{eq:fold}
\end{align}
It is now clear that $\Phi_{k,N}(\cS_{k,N}) \subset \{0,1,\ldots,N/k\} \times \cS_{k-1,N}$ and that $\Phi_{k,N}$ is piecewise affine, with integer matrices and translations. We now describe its inverse. Let
{
\begin{align}
\D_{k,N,m}=\bigg\{y\in\{0,1,\ldots,N/k\}\times\cS_{k-1,N}:{}&
 \sum_{\nu=2}^{i+1}y_\nu<y_1+\frac{(i-1)N}{k}\ (1\leq i<m),
 \nonumber\\[-1mm]
&\sum_{\nu=2}^{m+1}y_\nu\geq y_1+\frac{(m-1)N}{k}  (m<k) \bigg\},
\label{eq:targetPieces}
\end{align}
for $m=1,\ldots,k$. We claim that $\Phi_{k,N}({\C_{k,N,m}}) = \D_{k,N,m}$ for each $m$ and that
\begin{align}
\Phi_{k,N}^{-1}(y)
&=
\begin{cases}
\big(y_1,y_2-y_1,y_3,\ldots,y_k\big),
& \text{if }m=1,\\[1mm]
\begin{aligned}[t]
&\Big(
y_m+\frac Nk+1,
y_1+\frac{(m-2)N}{k}-1-\sum_{\nu=2}^{m}y_\nu,
y_{m-1},\ldots,y_2,\\
& \quad \quad \sum_{\nu=2}^{m+1}y_\nu-y_1-\frac{(m-1)N}{k},
y_{m+2},\ldots,y_k\Big),
\end{aligned}
& \text{if }2\leq m \leq k.
\end{cases}
\label{eq:foldInverse}
\end{align}
We prove that in the following proposition.

\begin{proposition}[The fold]\label{prop:fold}
{Let $k\geq2$ and let $N\geq1$ be divisible by $k$.}
The map
\begin{align*}
 \Phi_{k,N}:\cS_{k,N}
 \longrightarrow
 \{0,1,\ldots,N/k\}\times
{\cS_{k-1,N}}
\end{align*}
induces an $\aff_k(\Z)$-equidecomposition of $\cS_{k,N}$ and $ \{0,1,\ldots,N/k\}\times \cS_{k-1,N}$ given by \eqref{eq:sourcePieces} and \eqref{eq:targetPieces} (respectively) 
of complexity bounded by ${2k+1}$.
{Its inverse is given by \eqref{eq:foldInverse}.}
\end{proposition}

{
\begin{proof}
{Let's} first verify that $\Phi_{k,N}({\C_{k,N,m}}) \subset \D_{k,N,m}$. Let $x\in\C_{k,N,m}$ and put $y=\Phi_{k,N}(x)$. 
Direct substitution in \eqref{eq:fold} gives
\[
\sum_{\nu=2}^{i+1}y_\nu-y_1-\frac{(i-1)N}{k}
=\begin{cases}
\displaystyle\frac{(m-i)N}{k}-\sum_{\nu=1}^{m-i}x_\nu<0,
 &1\leq i\leq m-2,\\[1mm]
-x_2-1<0,&i=m-1,\\[1mm]
x_{m+1}\geq0,&i=m<k.
\end{cases}
\]
These are exactly the inequalities defining $\D_{k,N,m}$. Conversely, let $y\in\D_{k,N,m}$ and define $x=\Phi_{k,N}^{-1}(y)$ (which {a priori} is not the inverse map).  {It is not hard to verify that $\Phi_{k,N}^{-1}(\{0,1,\ldots,N/k\}\times\cS_{k-1,N})\subset\cS_{k,N}$ and that the $\D_{k,N,m}$ partition $\{0,1,\ldots,N/k\}\times\cS_{k-1,N}$.} The case $m=1$ is trivial.
Suppose $m\geq2$. Then $x_1=y_m+N/k+1>N/k$.
For $2\leq i\leq m$,
\begin{align}
\sum_{\nu=1}^{i}x_\nu
=y_1+\frac{(m-1)N}{k}
-\sum_{\nu=2}^{m-i+1}y_\nu > \frac{iN}{k}
\label{eq:recoveredSums}
\end{align}
when $i<m$. If $i=m$, the sum on the right is
empty, so the result is at most $mN/k$ because $y_1\leq N/k$.
This proves $x\in\C_{k,N,m}$.

Substitution in the two displayed formulas gives
$\Phi_{k,N}(x)=y$ and $\Phi_{k,N}^{-1}(\Phi_{k,N}(x))=x$
on the respective pieces. Each pair of affine formulas is then invertible
on $\RR^k$. Both have integer coefficients because $k\mid N$,
so their linear parts belong to $\GL_k(\Z)$.
Finally, there are at most $k$ nonempty source pieces, each piece is defined by at most ${2k+1}$ affine inequalities,  their defining
normals have $\ell^1$-norm at most $k$, and the matrices present in
\eqref{eq:fold} have entries in $\{0,1\}$. This proves the desired complexity
bound.
\end{proof}
}
}

\section*{Acknowledgments}
F.G. acknowledges support from the Office of Naval Research
GRANT14201749, the Serrapilheira Institute
(Serra-2211-41824), FAPERJ (E-26/210.245/2024)
and CNPq (309910/2023-4). M.R. acknowledges support of NSF grant DMS-2401106.

\section*{Use of generative AI}
We began working on Rudin's problem in late October 2025. By December 2025 we
had written notes on the boundedness of polynomials with $\sqrt{2}$
quadratic-phase coefficients and knew how to orthogonalize them using
maps arising from the equidecomposition of the simplex into rectangular
boxes via unimodular simplices.
By May 2026 we had explicit equidecompositions in dimensions 3 and 4. In the
summer of 2026 we obtained the full explicit equidecomposition construction
(second method in Section~5) and square-root cancellation over balanced boxes
(Section~3). An AI system supplied one lemma in Section~3, namely that
$\overline{p}/q$ is badly approximable with denominators below $q$ when $p/q$
is a convergent to a badly approximable irrational. At that point we had two
worked-out drafts covering the two extreme cases: one proving a sharp
square-root cancellation estimate on balanced boxes with arbitrary linear
inequalities as additional constraints, the other bounding quadratic
exponential sums with no polytope constraints beyond $n_1+\dots+n_K=N$.
We gave both drafts to an AI system and asked it to bridge them; it produced
the statement and a first proof of the decomposition theorem (Section~4).
We formalized that proof in Lean with AI assistance, completing the
formalization by 25~August, and subsequently rewrote and simplified the
proof repeatedly; the version in the paper is our own. A complete 69-page
draft of the proof, for $N$ divisible by $(K-1)!$, existed by 30~August 2026,
and the general case followed on 1~September using Olevskii's absorption
trick, whose use we had anticipated from the beginning of the project.
Throughout, we used AI tools to typeset, edit, and check drafts of our
arguments. The paper was submitted to arXiv on 9~October 2026; our work was
independent of Mark Lewko's construction, registered on Palomar the same
day (PALOMAR-2026-10-09-000004). After submission to the arXiv we
formalized the paper in Lean~4 with AI assistance, following the paper
outline; the formalization is available at
\url{https://github.com/maksym-radziwill/Rudin}. The formalization
deviates from the paper in two respects: the inner product is defined by
\eqref{def:innerprod} rather than as an integral over the sphere; and of
the two proofs of Proposition~\ref{prop:UNmap} in
Section~\ref{sec:explicitEquidecomposition}, only the constructive one is
formalized. The statement of each theorem, proposition and lemma
appearing in our paper is formalized as stated, except for
Proposition~\ref{prop:shiftedONS} and
Lemma~\ref{lem:facenorm}, which are formalized in a weaker form.

The decomposition algorithm in Section~\ref{sec:decomp} has parallels
with methods from integer programming. Its recursion
along thin directions resembles Lenstra's algorithm~\cite{Len}.
Turning an integral direction into a coordinate and organizing cases
uniformly over the right-hand sides also occur in parametric integer
programming~\cite{Kan92,ES08}. Here the changes of variables must preserve
the diagonal quadratic form, and the decomposition must control the
side lengths of the enclosing boxes. We did not find
Theorem~\ref{thm:decomp} in this form in the literature.

\end{document}